\documentclass{amsart}
\usepackage{graphicx} 

\usepackage{comment}

\usepackage{longtable}
\usepackage[acronym]{glossaries}
\usepackage[dvipsnames]{xcolor}

\newcommand{\bea}{\begin{eqnarray}}
\newcommand{\eea}{\end{eqnarray}}
\newcommand{\be}{\begin{equation}}
\newcommand{\ee}{\end{equation}}
\newcommand{\tensor}{\otimes}
\newcommand{\Hom}{\mathrm{Hom}}
\newcommand{\Atn}{A^{\tensor n}}
\newcommand{\Atm}{A^{\tensor m}}
\newcommand{\pri}{^{\,\prime}}
\newcommand{\inv}{^{-1}}
\newcommand{\eps}{\varepsilon}

\newcommand{\id}{\mathrm{id}}
\newcommand{\eul}{\mathbf{e}}
\newcommand{\bige}{\mathbf{E}}

\newcommand{\ol}{\overline}
\newcommand{\cob}{$\mathbf{2Cob}$\,}
\newcommand{\mgnbar}{\overline{\mathcal{M}}_{g,n}}
\newcommand{\Mbar}{{\overline{\mathcal{M}}}}
\newcommand{\la}{{\langle}}
\newcommand{\ra}{{\rangle}}

\usepackage{amsmath}
\usepackage{amsthm}
\usepackage{amsfonts}
\usepackage{bbm}
\usepackage{amssymb}
\usepackage{geometry}
\usepackage{tikz}
\usepackage{tikz-cd}
\usepackage{txfonts}
\usepackage[inline, shortlabels]{enumitem}
\usepackage{blkarray}
\usetikzlibrary{arrows,patterns,decorations}
\usepackage{titleref}
\usepackage[normalem]{ulem} 
\usepackage{float}
\usepackage{graphicx}
\usepackage{dcolumn}
\usepackage{mathptmx}
\usepackage{etoolbox}
\usepackage{xcolor}
\usepackage[caption=false]{subfig}
\usepackage{tabularx}
\usepackage{hyperref}

\newtheorem{Lem}{Lemma}[section]
\newtheorem{Thm}[Lem]{Theorem}
\newtheorem{Cor}[Lem]{Corollary}

\newtheorem{Prop}[Lem]{Proposition}

\theoremstyle{definition}
\newtheorem{Ex}[Lem]{Example}
\newtheorem{Def}[Lem]{Definition}
\newtheorem{Rmk}[Lem]{Remark}
\newtheorem{Quest}[Lem]{Question}

\title{Almost TQFTs via colored ribbon graphs}
\author{William Davis and Olivia Dumitrescu}
\date{March 2026}
\address{University of North Carolina at Chapel Hill, 340 Phillips Hall, CB 3250 NC 27599-3250 email:dolivia@unc.edu, willdavismath@gmail.com}

\usepackage[ maxbibnames=10, backend=biber]{biblatex}
\defbibheading{References}{References}

\begin{document}


\begin{abstract}
In this paper, we introduce \textit{ribbon TQFTs} via Edge Contraction/Construction Axioms of colored ribbon graphs as an extension of the 2D TQFT axioms for ribbon graphs formulated in \cite{DM_ribbon}. We investigate nearly Frobenius structures and Almost TQFTs defined in \cite{GLSU} together with ribbon TQFTs. We give a classification result for ribbon TQFTs that extends the one obtained for Frobenius algebras in \cite{DM_ribbon}. In particular, the Edge Contraction/Construction Axioms of colored ribbon graphs in this work become equivalent to the functorial Axioms of TQFTs governed by the sewing principle of Atiyah and Segal discussed  in \cite{DaDu} and \cite{GLSU}. As an application, we obtain that the recursion of generalized Catalan numbers can be twisted by Almost TQFT for co-unital nearly Frobenius algebra.



\end{abstract}

\keywords{ Almost TQFT, Frobenius structures, nearly Frobenius algebras, 2D TQFT, Colored ribbon graphs}
\subjclass[2020]{Primary: 14C20 Secondary:  14E05, 14E30, 14J45}

\maketitle
\tableofcontents

\section{Introduction}

In \cite{DMSS} the authors introduced {\it generalized Catalan numbers}
as a count of a particular type of ribbon graphs that satisfy a recursion relation similar to the {\it Cut and Join Equation of Hurwitz numbers}, we review the Catalan recursion in Section \ref{Catalan}. The formalism of Topological Recursion in \cite{DMSS} implies that the generating function of the Catalan numbers is related to the Gromov-Witten invariants of a point \cite{W1991}. These invariants related to the intersection theory on the moduli space of stable curves  $\overline{\mathcal{M}_{g,n}}$,  were first computed by Kontsevich in \cite{Kontsevich_intersection} via the theory of trivalent ribbon graphs and matrix integrals, solving a conjecture of Witten \cite{W1991}. The recursion relation of \cite{DMSS},  is one example of an enumerative problem known to give rise to the intersection numbers of $\overline{\mathcal{M}_{g,n}}$, the Deligne-Mumford moduli space of stable algebraic curves of genus $g$ and $n$ distinct marked points. 
 In an influential paper of Harer and Zagier \cite{HZ} the authors compute the Euler characteristic of $\mathcal{M}_{g,n}$, using a recursion formula (page 475) that is closely related to the recursion of the Catalan numbers of \cite{DMSS}.
 
 There have been many developments in the theory of Hurwitz numbers  and its connection to topological recursion, ELSV formulas, or intersection numbers of $\overline{\mathcal{M}_{g,n}}$, etc. (see for example  \cite{Sh,EMS,OP}). We recall that the orbifold Hurwitz numbers can also be computed via a count of ribbon graphs with additional information of points attached at each vertex (see \cite{DM_Hurwitz}). The Edge Contraction axioms applied to this particular count of ribbon graphs recover the Cut and Join Equation for orbifold Hurwitz numbers. While restricting to genus 0 and 1 marked point, the recursion relation of ribbon graphs becomes a count of trees and gives rise to  the spectral curve of the Hurwitz numbers, called the Lambert curve \cite{DM_Hurwitz}, the computation is similar to the one described in Section \ref{Catalan} for Catalan numbers. 


 In enumerative geometry some important invariants are related to the intersection of cohomology classes on moduli space of curves. 
For example, Hurwitz numbers can be expressed in terms of correlators of the cotangent classes on $\overline{\mathcal{M}_{g,n}}$,
and the total Chern class of the rank $g$ Hodge bundle  $c(\mathbb{E})\in H^*(\overline{\mathcal{M}_{g,n}},\mathbb{Q})$.  Similarly the Weil Petersson volume of Mirzakhani's work \cite{M_inventiones,M_jams} can be expressed in terms of the correlators between the cotangent classes and the exponential of the kappa class, $exp(\kappa_1)\in H^*(\overline{\mathcal{M}_{g,n}}, \mathbb{Q})$. These numbers satisfy a recursion relation that encodes in their generating functions Gromov-Witten invariants of the point. This shows an interesting similitude between {\it the Hurwitz numbers,  the Weil Petersson volumes and the Catalan numbers}. We recall that the total Chern class of the Hodge bundle, $c(\mathbb{E})$ and the exponential of the kappa class $exp(\kappa_1)$ are examples of cohomological field theories. It is an interesting question to describe the  Catalan numbers $C_{g,n}(\mu_1,\dots,\mu_n)$ in terms of correlators of
cotangent classes on $\overline{\mathcal{M}_{g,n}}$ and a cohomological field theory  on the moduli space of $n$-pointed stable curves of genus $g$.


The ribbon graphs operations induced by pairs of pants decomposition of Riemann surfaces can further be used to define maps between moduli spaces of curves, and give an enumerative approach to the classification results for Cohomological Field Theories. Namely, in \cite{DM3} (in progress), the authors extended the ribbon TQFT formalism for Frobenius algebras to a formalism for Cohomological Field Theories of \cite{KM} using the theory of ribbon graphs. This combinatorial approach gives rise to a classification result  for Cohomological Field Theories defined on semi-simple Frobenius algebras,
 similar to Theorem \ref{thm B}, using the formalism of ribbon graphs developed by Thurston et al. in \cite{STT}, Strebel in \cite{Strebel},  Mulase and Penkava in \cite{MP} and others.

It is an interesting problem to generalize Theorem \ref{thm:coloredcatalan}, by deriving a Catalan recursion twisted by Cohmological Field Theories and further relate it to enumerative geometry. One motivating question in this direction {\it can the twisted Catalan numbers by cohomological field theories be interpreted in the topological recursion formalism to give information on the enumerative geometry and intersection numbers of moduli spaces?}

\vskip 10pt

This paper is organized as follows. In Section \ref{Catalan} we recall the recursion of generalized Catalan numbers and its connection with Gromov Witten invariants of a point. In Section \ref{Frob section} we offer a generous introduction to Frobenius algebras, Nearly Frobenius algebras,  2D TQFTs, we summarize the main steps for the classification of 2D TQFTs and we state Theorem \ref{theorem A}, the general classification result for Almost TQFTs recently given by the authors in \cite{DaDu}. We recall that the main classification result TQFTs via ribbon graphs in \cite{DM_ribbon} leverages the isomorphism between a Frobenius algebra and its dual. However, the difficulty in extending this result to Nearly Frobenius algebras is that there is no longer an isomorphism between the algebra and its dual, so instead we need to modify the methodologies used to obtain that result.

In \cite{DM_ribbon}  the authors introduce a set of axioms in which {\it contracting a straight edge} of a cell graph acts like {\it multiplication} on an algebraic structure and {\it contracting a loop} at a single vertex acts like {\it comultiplication}. In this work we adopt these axioms, and we further introduce additional ones for colored cell graph in Section \ref{SectionCCG}. In Section \ref{SecInfiniteCCG}, we apply colored cell graph axioms to potentially infinite-dimensional nearly Frobenius algebras.

More precisely, in Section \ref{SectionCCG} we introduce the Ribbon TQFT for Nearly Frobenius algebras,  Definition \ref{def:ribbonTQFT} via a set of axioms of  four types,  Definitions \ref{def:redaxioms}, \ref{Output edge-construction axioms},  and \ref{def:flowaxioms}. The axioms of the first type are similar to those used on the Frobenius algebra in the original classification result, those of the second type act as an analog of the first type for elements of the dual on the output side of the TQFT, and the axioms of the third type are inspired by topological consequences in the connections between the first two. The fourth axiom is needed for Nearly Frobenius algebras which no longer have a unit, while for Frobenius algebras it is implied by the first two types of axioms and the existence of the unit (see Lemma \ref{lem:redloopremoval}).


We recall the technique of pairs of pants decomposition for Riemann surfaces explored by Eynard and Orantin \cite{EO_communications,EO_preprint}, and previously by Mirzakhani \cite{M_inventiones,M_jams} is a fundamental tool serving as a foundation for the framework of topological recursion. Recall that a cell graph can be thought of as a graph embedded on such a Riemann surface, so the topological recursion formalism is what serves as the inspiration for the axioms of the first two types, depending on whether this decomposition occurs at boundary components considered to be ``inputs'' or ``outputs'' of the graph, which we color by red and blue respectively.

\begin{figure}[ht]
    \centering
	\includegraphics[width=0.3\linewidth]{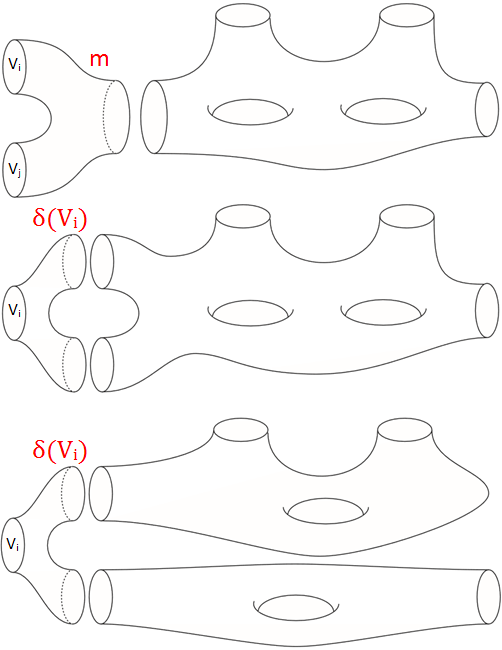}
	\caption{Pairs of pants decomposition of Riemann surfaces leading to topological recursion}
	\label{fig:toprecursion}
\end{figure}

  The flow axioms on colored ribbon graphs we give in Section \ref{def:flowaxioms} were inspired by graph techniques introduced by Kontsevich et al. in \cite{KTV} where the authors use ribbon graph operations to describe moduli spaces related to a construction of Lurie \cite{Lurie}.

In Section \ref{Section 6} we focus on Frobenius algebras and prove a classification result for Ribbon TQFT, Theorem \ref{thm:cellgraphclassification}, extending the main result of paper \cite{DM_ribbon} (see also \cite{Dumitrescu_report,DM_invitation}). In Section \ref{SecInfiniteCCG} we prove the classification result for Almost TQFTs introduced in this paper, for Nearly Frobenius Algebras, Theorem \ref{thm:infcellgraph}. More explicitly,

\begin{Thm}\label{thm B}
    Let $A$ be a nearly Frobenius algebra with coproduct $\delta$ and the Euler map $\bige$. Then the {\it ribbon TQFT} defined in \ref{def:ribbonTQFT} associated to a colored cell graph $\gamma$ of type $(g,n,m)$ with $n,m\geq 1$, is independent of the graph and is given by  
    \be
    \Omega(\gamma)(v_1,\ldots,v_n)=\delta^{m-1}(\bige^g(v_1\cdots v_n)).
    \ee
    Where $\delta^0$ is the identity map (for the case where $m=1$).
\end{Thm}
Notice that even in the finite dimensional case, there exist nearly Frobenius algebra structures which are not counital, see Example \ref{ex:finiteNFA}. Therefore, Theorem \ref{thm B} provides a more general classification result than was obtained in \cite{DaDu}, provided that we impose one additional constraint on the topological type of the cell graph.
\begin{Thm}\label{prop:1.2}
    Let $A$ be a counital nearly Frobenius algebra with counit $\eps$, Euler map $\bige$, and $m=0$. Consider $\gamma$ a colored cell graph of type $(g,n,0)$ with $n\geq 1$, and $\Omega(\gamma)\in\Hom(A^{\tensor n},K)$ an assignment of a multilinear map satisfying the Edge Contraction Axioms defined in \ref{SecRedAxioms}. The ribbon TQFT associated to the nearly Frobenius algebra $A$ is independent of the graph and is given by
    \be
    \Omega(\gamma)(v_1,\ldots,v_n)=\eps(\bige^g(v_1\cdots v_n)).
    \ee
\end{Thm}

\noindent We note that the axioms in Definition \ref{SecRedAxioms} precisely recover the axioms of \cite{DM_ribbon} in the case that $m=0$. Therefore, the proof of Theorem \ref{prop:1.2} follows the same arguments of \cite{DM_ribbon} with the additional consideration of the Remark \ref{rmk:loopremoval}, ie the disc-bounding loop removal, Case (4) of the Ribbon TQFT Definition \ref{def:ribbonTQFT}. The idea of this proof uses the graph independence of the edge contraction operations and hinges on the following process: we contract all straight edges of a cell graph until we are left with a graph of one vertex, and some number of loops. The loops are either loops that bound a disc or those that enclose some positive genus. We may remove each disc bounding loop, and the contraction of the remaining loops corresponds to applying the Euler map once for each genus captured by the remaining loops.\\
\\
\noindent Further, we recall Corollary 1.3 of \cite{DaDu} stating that for counital nearly Frobenius algebras, the classification of ribbon graph TQFTs given by Theorems \ref{thm B} and \ref{prop:1.2} is the same as the classification of Almost TQFTs in Theorem 1.2 of \cite{DaDu}. However, Theorem \ref{thm B} gives a more general classification of ribbon graph TQFTs in the case where a counit may not exist.\\
\\
\noindent In Section \ref{twisted Catalan} we combine results proved in this paper to generalize the Catalan recursion twisted by Almost TQFT, for Nearly Frobenius Algebras.  More precisely, the Catalan recursion \ref{eqn:catalanrecursion} can be twisted by Almost TQFT \ref{twistedtqft}. We refer to Theorem \ref{thm:coloredcatalan} for an explicit recursion formula.

\begin{Thm}\label{thm:nearlycatalantwist}
For $A$  a counital nearly Frobenius algebra with counit $\eps$, Euler map $\bige$, then the generalized Catalan recursion can be twisted by the Almost TQFT. 
	
\end{Thm}

\noindent Note that if $A$ is a finite-dimensional Frobenius algebra, then it is also a counital nearly Frobenius algebra. In particular, Theorem \ref{prop:1.2}  extends the result obtained in \cite{DM_ribbon} to a result for counital nearly Frobenius algebras.


To motivate our interest in Theorem \ref{thm:nearlycatalantwist}, we consider the following example. For $G$ a finite group, let $A=Z(C[G])$ be the center of the group algebra as a Frobenius algebra representing the orbifold cohomology of  the classifying stack of $G$, $BG$. If $G$ is a finite group, then $A$ is finite dimensional with a basis indexed over conjugacy classes of $G$. It was proved in \cite{Serrano, JK} that the recursion relations of decorated dessins and twisted Catalan numbers via the 2D TQFT corresponding to the Frobenius algebra $A=H^*(BG)$ give information on the cotangent class intersection numbers on the moduli space 
$\overline{\mathcal{M}_{g,n}}(BG)$, via the methods exposed in Section \ref{Catalan}. 
This observation was the starting point for the work started in  \cite{DM_invitation, DM_ribbon, DMSS} to compute recursive relations of twisted Catalan numbers in the attempt of obtaining information of the intersection theory of more general moduli spaces of stable maps.  We leave it as an open question whether the result of Theorem \ref{thm:nearlycatalantwist} can be applied to topological invariants similar to the results of the twisted recursion obtained in \cite{DM_invitation}. We end the Introduction  by the following observation:
\begin{Cor}
	If $A$ is a finite-dimensional Frobenius algebra then Theorem \ref{thm B} and Theorem \ref{prop:1.2} complete the classification of \cite{DM_ribbon} for $m\geq 1$. 
\end{Cor}
\noindent Indeed, notice that in the case that $A$ is a finite-dimentional Frobenius algebra (and therefore also a nearly Frobenius algebra) the Euler map is just multiplication by the $g$th power of the Euler element, $\bige^g(v)=e^gv$.

At the end of this work we further introduce {\it colored Catalan numbers} and we leave it as an open question whether they satisfy a recursion relation that can be twisted by Almost TQFTs for nearly Frobenius algebras via the use of the Theorem \ref{thm:coloredcatalan}.

Lastly, we would like to mention interest in recent developments of Garoufalidis, Scholze, Wheeler and Zagier \cite{GSWZ} in the theory of Habiro rings of a number field and previous results of \cite{GZ}. In particular, the spectral curve of the Catalan numbers was recovered from the asymptotic expansions of Nahm sums, solutions to q-difference equations. Investigations regarding comparisons of q-difference equations and asymptotic expansions of Nahm sums versus the quantum curve of Catalan numbers and its asymptotic expansions are further topics of interest.

\subsection*{Acknowledgements} The authors would like to express their graditude to the Max Planck Institute for Mathematics, Bonn for their generosity, hospitality and stimulating environment during their long stay in 2023, when this project was initiated. They would also like to express their gratitude to the IHES Bures sur Yvette and RIMS Kyoto for the 2024 visits.  They would also like to thank  David Rose and  Luuk Stehouwer for useful suggestion and multiple discussions related to this project at UNC Chapel Hill and MPIM. The authors research was partially supported by the NSF-FRG DMS 2152130 grant and the UNC JFDA Award 2022.

 During their stay in MPIM, Bonn the second author gave a mini-course at University of Oxford, where some preliminaries ideas of this project were presented; this work is dedicated to the Oxford Lectures Volume. In the Oxford lectures the author also talked about the orbifold Hurwitz numbers via ribbon graphs,  Mirzakhani's work on Weil Petersson volumes and their comparison with the Catalan numbers; the Cohomological Field Theory approach via ribbon graphs, that motivated this work.

 Shortly before String Math 2016, Maryam Mirzakhani expressed an interest in the broad landscape of topological recursion, and we suspect that the recursion of the Catalan numbers is the problem in our research in enumerative geometry closest to her work. We dedicate this contribution to her memory.

\section{Enumerative Geometry of Catalan numbers}\label{Catalan}
We expose this Section based on the results obtained in \cite{DMSS}, \cite{DM_invitation}. In \cite{DMSS} the authors presented the enumerative geometry that we expose in this Section from the point of view of weighted count of clean Belyi morphisms, that is closely related to the lattice point count of Norbury and Do of \cite{Norbury}. We consider the Catalan recursion to offer an easier apporach to the enumerative geometry of moduli spaces of curves that we will further use in Theorem \ref{thm:coloredcatalan}.

Throughout the paper, we will be using the theory of cell graphs,  ie the Poincar\'e dual to ribbon graphs, first introduced by 't Hooft \cite{Hooft} in physics and Grothendieck \cite{Grothendieck} in mathematics. This theory was used by Kontsevich \cite{Kontsevich_intersection} to give a first approach to the computation of Gromov-Witten invariants. 
\begin{Def}
    (Cell graph) A \textbf{cell graph} of topological type $(g,n)$ is the 1-skeleton (the union of 0-cells and 1-cells) of a compact oriented topological surface of genus $g$ with $n$ marked 0-cells. We refer to a 0-cell as a \textit{vertex}, a 1-cell as an \textit{edge}, and a 2-cell as a \textit{face}. Each edge is divided into two \textit{half-edges} connected at the midpoint of each edge. The set $\Gamma_{g,n}$ denotes the set of all cell graphs of topological type $(g,n)$.
\end{Def}
\begin{Rmk}
    Cell graphs are the dual notion of what are known as \textit{ribbon graphs} or \textit{fatgraphs}, introduced by Grothendieck \cite{Grothendieck} as \textit{dessin d'enfants}. A ribbon graph is a graph with an assigned cyclic order to the incident half-edges at every vertex. This ordering then induces a cyclic order to incident half-edges at every face. In this way, any cell graph is also a ribbon graph, though ribbon graphs are typically labeled by their faces and cell graphs are labeled by their vertices.
\end{Rmk}
\begin{Rmk}
Because of the previous remark, we can view a cell graph as three pieces of information: a list of vertices, a list of edges, and at each vertex a cyclic order of the half-edges incident to that vertex. We then consider embedding this graph onto a compact oriented topological surface where the $n$ vertices correspond to $n$ marked discs cut out of the surface. We embed the graph onto the surface in a way such that at each vertex, the specified cyclic order of its half-edges appears clockwise on the surface. In this way, the topological genus $g$ is the smallest non-negative integer such that this graph can be embedded on a surface of genus $g$ without any edges crossing.
\end{Rmk}

Consider the set $\Gamma_{g,n}(\vec{\mu})$ to be the set of all uncolored cell graphs of type $(g,n)$ with specified degrees\\ $\vec{\mu}=(\mu_1,\ldots,\mu_n)$ on the vertex set. If we attempt to answer enumerative problems with this set regarding the size of it, we run into some trouble of overcounting where a cell graph may have some non-trivial automorphisms with itself in a way that fixes edges and vertices but permutes faces. We could account for this overcounting by automorphisms, or we could avoid the automorphism issue altogether by instead considering \textit{arrowed cell graphs}. An arrowed cell graph is a cell graph in which at each vertex we mark one of the half-edges at that vertex and consider that half-edge to be the starting point of that cyclic order. Let $\widehat{\Gamma}_{g,n}(\vec{\mu})$ be the set of arrowed cell graphs of type $(g,n)$ with specified vertex degrees given by $\vec{\mu}$. Now consider
$
C_{g,n}(\vec{\mu}):=|\widehat{\Gamma}_{g,n}(\vec{\mu})|$ to be its cardinality. 
It is easy to observe that $C_{0,1}(2k)$ numbered of arrowed cell graphs on a Riemann sphere with one vertex has even index at the vertex since edges can only be loops attached to the vertex. An arrowed cell graph of type $(0,1)$ further corresponds to a unique pair of $k$ parenthesis, as we start placing parentheses at the arrow  following the orientation of the surface

These numbers $C_{g,n}(\vec{\mu})$ give a generalization of the classical Catalan numbers since
\be\label{cat}
C_{0,1}(2k)=\dfrac{1}{k+1}\dbinom{2k}{k}=C_k
\ee
\noindent is the $k$-th traditional Catalan number.

The numbers $C_{g,n}(\vec{\mu})$ then satisfy a recursion relation similar to the \textit{Cut and Join} which recover \ref{cat recursion} for the case where $g=0$ and $n=1$. We will simplify the notation for the Catalan recursion and use this simplified version later on in this work. Note here that the notation $\vec{u}_{[n]\setminus \{1,j\}}=(\widehat{u}_1,u_2,\ldots,\widehat{u}_j,\ldots,u_n)$ means we omit both the first and $j$th components of the vector $\vec{u}$ and $C_{g,n-1}(\mu_1+\mu_j-2,\vec{\mu}_{[n]\setminus\{1,j\}})=C_{g,n-1}(\mu_1+\mu_j-2,\mu_2,\ldots, \widehat{\mu_j}, \ldots, \mu_n)$ means we replace the first component of $\vec{\mu}$ with $\mu_1+\mu_j-2$ and omit the $j$th component. We obtained the following recursion formula of the generalized Catalan numbers \cite{DMSS,DM_invitation} 
\be\label{eqn:catalanrecursion}
\begin{aligned}
	C_{g,n}(\vec{\mu})=&\sum_{j=2}^n \mu_jC_{g,n-1}(\mu_1+\mu_j-2,\vec{\mu}_{[n]\setminus\{1,j\}})\\&+\sum_{\alpha+\beta=\mu_1-2}\Biggl[C_{g-1,n+1}(\alpha,\beta,\vec{\mu}_{[n]\setminus1})+\sum_{\substack{g_1+g_2=g\\I\sqcup J=\{2,\ldots,n\}}}C_{g_1,|I|+1}(\alpha,\vec{\mu_I})C_{g_2,|J|+1}(\beta,\vec{\mu_J})\Biggr],
\end{aligned}
\ee
and it was shown that this equation can be used to compute Gromov-Witten invariants of a point via a change of coordinates. Each term in this expression corresponds to contracting the marked edge at the first vertex to result in a different cell graph. The first term corresponds to contracting a straight edge connecting vertices $1$ and $j$, the second term corresponds to contracting a loop at vertex $1$ that leaves the graph connected, and the third term corresponds to contracting a loop at vertex $1$ that disconnects the graph.\\

Restricting the recursion of generalized Catalan numbers, Equation \ref{eqn:catalanrecursion}, to the genus 0 and 1 marked point case, via the observation of Equation \ref{cat}, we obtain the celebrated recursion that the classical Catalan numbers satisfy
\be\label{cat recursion}
C_k=\sum_{a+b=k-1} C_a \cdot C_b
\ee

It is an undergraduate exercise in discrete mathematics textbook to show that
the recursion relation of classical Catalan numbers $C_k$, \ref{cat recursion}, is equivalent to the equation of an algebraic curve. Indeed, taking the generating function of these numbers $z(x)=\sum_{m=0}^{\infty} C_{m}x^{-2m-1}$ 
 and using the classical recursion\ref{cat recursion} we obtain the {\it spectral curve of the Catalan numbers}
\be\label{cat curve}
z^2-xz+1=0
\ee
 that is fundamental in the topological recursion formalism.

Consider the generating function for the generalized Catalan numbers is defined as

\be
\label{F}
F_{g,n}(x_1,\dots,x_n)
:= \sum_{\mu_1,\dots,\mu_n>0}\frac{C_{g,n}(\mu_1,\dots,\mu_n)}
{\mu_1\cdots\mu_n}x_1^{-\mu_1}\cdots 
x_n^{-\mu_n}.
\ee

The spectral curve of the Catalan numbers \eqref{cat curve} is seen as a local expression on one coordinate chart of a singular curve inside the Hirzebruch surface $\mathbb{F}_2$. As in \cite{DMSS} we perform  a change of coordinates that corresponds to the normalization coordinate of the blown-up curve 

\be\label{xt}
x = x(t) = \frac{t+1}{t-1}+\frac{t-1}{t+1},
\ee
For each $(g,n)$, in the stable range $2g-2+n>0$, then $F_{g,n}\big(x(t_1),x(t_2),\dots,x(t_n)\big)
$ becomes a Laurent polynomial. 

There exists a differential operator
that annihilates the generating function of the free energies of the generalized Catalan numbers. We call it {\it the quantum curve of Catalan numbers}. The spectral curve used in the Topological Recursion formalism for the Catalan recursion, \eqref{cat curve} gives rise to its quantum curve $$\hbar \frac{d^2}{dx^2}+ \hbar x\frac{d}{dx} +1$$ where the Planck constant $\hbar$ is a deformation parameter.  In particular, the  quantization for the Catalan numbers was proved by Mulase and Sulkowski following a conjecture of Gukov and Sulkowski gives the following:

\begin{Thm} \label{thm MS} Let $\hbar$
	be a the Planck constant. Then
	\begin{equation}\label{quantum curve}
	\left(\hbar^2 \cdot \frac{d^2}{dx^2}+\hbar \cdot x \cdot \frac{d}{dx}+1\right) 
	\exp\left(\sum_{2g-2+n\ge -1}\frac{1}{n!}\cdot\hbar^{2g-2+n}\cdot F_{g,n}(t,\dots,t)
	\right)=0.
	\end{equation}
\end{Thm}

A surprising discovery was to find the intersection numbers of 
$\overline{\mathcal{M}_{g,n}}$ through the 
asymptotic expansion of solutions of the quantum curve. More precisely, the highest degree part of the Laurent polynomial $F_{g,n}$
is a homogeneous polynomial of degree
$6g-6+3n$
\be
\label{topdegree}
F^{\text{highest}}_{g,n}(
t_1,\dots,t_n)
=
\frac{(-1)^n}{2^{2g-2+n}}
\sum_{\substack{d_1+\dots+d_n\\
		=3g-3+n}}\la \tau_{d_1}\cdots \tau_{d_n}\ra_{g,n}
\prod_{i=1}^n \left(
|2d_i-1|!!\left(\frac{t_i}{2}\right)^{2d_i+1}
\right),
\ee

Here the cotangent class intersection  numbers on the moduli space $\overline{\mathcal{M}_{g,n}}$, denoted by 
\be
\label{intersection}
\la \tau_{d_1}\cdots \tau_{d_n}\ra_{g,n} =
\int_{\Mbar_{g,n}}
\psi_1 ^{d_1}\cdots \psi_{n}^{d_n}
\ee
are the coefficients of  $F^{\text{highest}}_{g,n}(t_1,\dots,t_n)$

These results of \cite{DMSS} further imply that
the derivatives $W_{g,n}=d_1\cdots d_n F_{g,n}$ 
satisfy the topological recursion based on 
the spectral curve \eqref{cat curve}, which is the
semi-classical limit of Equation \eqref{quantum curve}. In this work we will omit the details regarding the Eynard-Orantin formalism of Topological Recursion or WKB analysis, and rather focus on the mathematical aspects of TQFTs.

In the following sections, by exploiting both the theory of colored ribbon graphs  and the classifications of Almost TQFTs found in \cite{DaDu}, we obtain a recursion relation Theorem \ref{thm:coloredcatalan}, extending the standard recursion of \cite{DMSS}. It is a topic of further research to relate these new recursive relations to enumerative geometry via the topological recursion mechanism.


\section{Frobenius structures}\label{Frob section}

\subsection{Frobenius Algebras}\label{secFA}
\noindent Throughout this chapter, we will consider finite-dimensional, unital, commutative Frobenius algebras over a field $K$. We recall the following
\begin{Def}\label{def:FA} A be a {\it finite dimensional algebra} $A$, with a unit over a ground field $K$ equipped with a nondegenerate symmetric bilinear form $ \eta(u,vw)=\eta(uv,w),\quad \text{for\,all\,} u,v,w\in A$ is called a Frobenius algebra. We denote by $m$ its multiplication map $m: A\tensor A\to A$ and the bilinear form $\eta$ its {\it Frobenius form}.
\end{Def}

Let $\{e_1,\ldots, e_n\}$ be a basis for $A$ over the field $K$ and $1\in A$ be the unit of the Frobenius algebra. The Frobenius form $\eta$ induces a symmetric invertible matrix whose entries will be denoted by\\
$$\eta_{m,n}=\eta(e_m, e_n), \quad \eta=[\eta_{mn}] \quad \eta^{-1}=[\eta^{mn}].$$

We will further introduce the counit of the Frobenius algebra $\eps:A\to K$ as $\eps(v)=\eta(1,v)$ for $v\in A$, and $1\in A$ the multiplicative identity. In particular, the Frobenius form can be defined in terms of the counit 

$$\eta(u,w)=\eps(uw),\,\text{ for\,every \,} u,w\in A.$$

For a finite dimensional Frobenius algebras, the Frobenius form $\eta$ induces an isomorphism between $A$ and its dual $A^*$
\be\label{eqn:lambdaiso}
\lambda:A\longrightarrow A^*,\quad \lambda(u)=\eta(u,-).
\ee

The dual of the multiplication, $m^*$, and  the dualities $\lambda$ and $\lambda\tensor \lambda$ induce a unique map, $\delta$ called the comultiplication of $A$, that makes the diagram commutative

\be
	\begin{tikzcd}
		A \arrow[dd, "\lambda"'] \arrow[rr, "\delta", dashed] &  & A\otimes A \arrow[dd, "\lambda\otimes\lambda"] \\
		& \circlearrowleft &                                                \\
		A^* \arrow[rr, "m^*"']                                &  & A^*\otimes A^*                                
	\end{tikzcd}
\ee


On basis elements, the co-multiplication of $A$ can be defined: \be\label{eqn:deltabasis}
\delta(v):=\sum_{k,p,i,j=1}^{n}\eta(v,e_k e_p)\eta^{ki}\eta^{pj}e_i\otimes e_j\ee



\noindent Furthermore, assume that the Frobenius algebra is associative and co-associative in other words that the following relations hold $(m(m(p\otimes q)\otimes r)=m(p\otimes m(q\otimes r))$ and $(\delta\otimes \id)(\delta(w))=(\id\otimes\delta)(\delta(w))$ and for every $p,q,r,w \in A$. Then the maps $m$ and $\delta$ satisfy the ``Frobenius relation'' meaning that for all $u,w \in A$, 

\be\label{eqn:frobrel}\delta(uw)=(\id\tensor m)(\delta(u),w)=(m\tensor\id)(u,\delta(w)).\ee 

To see this observe first that any vector of a Frobenius algebra can be written in terms of a basis by the following expression

\be\label{eqn:canonbasis}
v=\sum_{a,b}\eta(v,e_a)\eta^{ab}e_b=\sum_{a,b}\eta(e_a,v)\eta^{ba}e_b.
\ee

Further, provided that $A$ is commutative and cocommutative, then
\begin{align*}
    \delta\circ m(u\tensor v)&=\delta(uv)\\
    &=\sum_{i,j,a,b}\eta(uv,e_ie_j)\eta^{ia}\eta^{jb}e_a\otimes e_b\\
    &=\sum_{i,j,a,b}\eta(ue_i,ve_j)\eta^{ia}\eta^{jb}e_a\otimes e_b\\
    &=\sum_{i,j,a,b,c,d}\eta(ue_i,e_c)\eta^{cd}\eta(e_d,ve_j)\eta^{ia}\eta^{jb}e_a\otimes e_b\\
    &=\sum_{i,j,a,b,c,d}\eta(e_i,ue_c)\eta^{cd}\eta(e_de_j,v)\eta^{ia}\eta^{jb}e_a\otimes e_b\\
    &=\sum_{j,b,c,d}\eta^{cd}ue_c\eta(e_de_j,v)\eta^{jb}\otimes e_b\\
    &=\sum_{j,b,c,d}\eta^{cd}\eta^{jb}\eta(v,e_de_j)ue_c\otimes e_b\\
    &=(m\otimes \id)(u\otimes \delta(v)).
\end{align*}Similarly, we have that
\begin{align*}
    \delta\circ m(u\tensor v)&=\delta(uv)\\
    &=\sum_{i,j,a,b}\eta(uv,e_ie_j)\eta^{ia}\eta^{jb}e_a\otimes e_b\\
    &=\sum_{i,j,a,b}\eta(ue_i,ve_j)\eta^{ia}\eta^{jb}e_a\otimes e_b\\
    &=\sum_{i,j,a,b,c,d}\eta(ue_i,e_c)\eta^{cd}\eta(e_d,ve_j)\eta^{ia}\eta^{jb}e_a\otimes e_b\\
    &=\sum_{i,j,a,b,c,d}\eta(u,e_ie_c)\eta^{cd}\eta(e_dv,e_j)\eta^{ia}\eta^{jb}e_a\otimes e_b\\
    &=\sum_{i,j,a,b,c,d}\eta(u,e_ie_c)\eta^{cd}\eta^{ia}e_a\otimes \eta(e_dv,e_j)\eta^{jb}e_b\\
    &=\sum_{i,a,c,d}\eta(u,e_ie_c)\eta^{cd}\eta^{ia}e_a\otimes (e_dv)\\
    &=(\id\otimes m)(\delta(u)\otimes v).
\end{align*}

As seen above, the composition $\delta\circ m:A\tensor A\longrightarrow A\tensor A$ plays an important role in the definition of a Frobenius algebra. This composition in the other order $m\circ \delta:A\longrightarrow A$ is important in the development of 2D TQFD.

The \textbf{Euler element} of a Frobenius algebra, denoted by $\eul\in A$, is defined by
\be
\eul = m\circ\delta(1).
\ee 
We can express the Euler element in terms of a basis
\be\label{eqn:eulbasis}
\eul = \sum_{a,b}\eta^{ab}e_ae_b.
\ee
    This element allows us to extend our understanding of 2D TQFT to those linear maps associated to surfaces cell graphs of positive genus.

\begin{Rmk}
    We also remark that for any vector $v$, then
    \be
    m\circ \delta(v)=\eul v.
    \ee
    This follows directly from \ref{eqn:deltabasis} and \ref{eqn:eulbasis}, we leave the proof to the reader.
\end{Rmk}

\begin{Ex}
    Consider the finite-dimensional polynomial algebra over $K$, $A=K[x]/(x^{n+1})$ with basis $\{1,x,x^2,\ldots,x^n\}$. Then $A$ may be endowed with many comultiplicative structures. Gonzalez et al. \cite{GLSU} showed that all possible coproducts on $A$ must be linear combinations of coproducts which take the form
    \be
    \delta_k(x^\ell):=\sum_{i+j=n+k+\ell}x^i \tensor x^j,\quad\mathrm{\,for\,} k\in\{0,1,\ldots,n\}.
    \ee
    We notice that for $k=0$, this $\delta_0$ is exactly the coproduct in Example 2.4 from \cite{DaDu} of the group algebra of the finite cyclic group of order $n+1$ generated by the element $x$, where the coproduct is induced by the counit $\eps:A\to K$ given by $\eps(1)=1$ and $\eps(x^\ell)=0$ for $\ell\neq 0$. The rest of these comultiplicative structures, however, do not arise from the counit of some Frobenius algebra, otherwise the Nearly Frobenius structure would not satisfy the Frobenius relation.
\end{Ex}

\subsection{Nearly Frobenius Algebras}
Any Frobenius algebra with both a unit and counit that has an associative multiplicative structure and coassociative comultiplicative structure is forced to be finite-dimensional. However, many results and classifications of Frobenius algebras may be interesting when extended to algebras that may be infinite dimensional but still maintain much of the same Frobenius structure. Because of this limitation, these new structures must either lack a unit or a counit (or potentially both). Our previous discussion of Frobenius algebras in \ref{secFA} showed how to define a Frobenius algebra based on its counit $\eps$, but Abrams \cite{Abrams} gave another formulation of Frobenius algebras based instead on the coproduct $\delta$. In fact, Abrams further showed that these two constructions are equivalent; any Frobenius algebra defined by its counit can equivalently be defined by its coproduct and vice-versa. Gonzalez et al. \cite{GLSU} used this construction via the coproduct to generalize the structure of Frobenius algebras to these similar algebras which may lack a counit.

\begin{Def}
    (Nearly Frobenius Algebra) Let $A$ be an algebra over a field $K$ endowed with a structure of multiplication $m$ and comultiplication $\delta$ where the product and coproduct maps satisfy the Frobenius relation \ref{eqn:frobrel}. Then, the collection of data $(A,m,\delta)$ is known as a \textbf{Nearly Frobenius algebra}.
\end{Def}
\begin{Rmk}
	We recall that the definition \ref{def:FA} requires a Frobenius algebra to be both unital and counital, but for a Nearly Frobenius algebra, there is neither such requirement.
\end{Rmk}

\begin{Def}
	For a Nearly Frobenius algebra, the \textbf{Euler Map,} $\mathbf{E}$ is defined as
	\be \bige:=m\circ\delta:A\longrightarrow A.\ee
\end{Def}

\noindent We recall that for a Frobenius algebra, the Euler element was defined as $\eul=(m\circ\delta)(1)$, and similarly $(m\circ\delta)(v)=\eul v$. However, for a Nearly Frobenius algebras without a unit, we cannot define $\eul$ as an element of the algebra. Therefore, this notion of the Euler map allows us to use a similarly useful construction in a new setting.
\begin{Rmk}
	For a Frobenius algebra, the Euler map is simply multiplication by the Euler element $e$.  $$\bige(v)=ev=m\circ \delta(v).$$
\end{Rmk}
The Euler map can be written in terms of a basis for $A$, 
\be\label{eqn:EULbasis}
\bige(v)=\sum_{i,j}\eta^{ij}ve_ie_j.
\ee
We then use this Euler map to classify values of the Almost TQFT for higher genus.

\begin{Ex}
	Any Frobenius algebra must also be a Nearly Frobenius algebra. The coproduct $\delta$ which arises as the dual of the multiplication map $m$ also serves as the coproduct of the Nearly Frobenius algebra structure.
\end{Ex}
\noindent It is, however, not necessarily the case that any Nearly Frobenius algebra arises from a Frobenius algebra structure. In fact, even in the case that a Nearly Frobenius algebra is finite-dimensional, it is not essential that it is also a Frobenius algebra. One such example of this phenomenon arises from the fact that even though a Frobenius algebra must have a unique counit, a Nearly Frobenius algebra may have many potential coproducts which satisfy the necessary comultiplicative structure, but do not arise as the dual of a product map under this unique counit.

\begin{Ex}\label{ex:finiteNFA}
    Consider the finite-dimensional polynomial algebra over $K$, $A=K[x]/(x^{n+1})$ with basis $\{1,x,x^2,\ldots,x^n\}$. Then $A$ may be endowed with many comultiplicative structures. Gonzalez et al. \cite{GLSU} showed that all possible coproducts on $A$ must be linear combinations of coproducts which take the form
    \be
    \delta_k(x^\ell):=\sum_{i+j=n+k+\ell}x^i \tensor x^j,\quad\mathrm{\,for\,} k\in\{0,1,\ldots,n\}.
    \ee
    We notice that for $k=0$, this $\delta_0$ is exactly the coproduct in the group algebra of the finite cyclic group of order $n+1$ generated by the element $x$, where the coproduct is induced by the counit $\eps:A\to K$ given by $\eps(1)=1$ and $\eps(x^\ell)=0$ for $\ell\neq 0$.\\
    \\
    \noindent However, for $k\neq0,$ this defines $(A,\delta_k)$ as a nearly Frobenius algebra without being a Frobenius algebra, despite being finite-dimensional. Indeed, notice that any Frobenius algebra must satisfy the relation on sewing discs that $(\eps\tensor\id)(\delta(v))=v$ for all $v$. In the case that $k\neq 0$, we see that \be
    (\eps\tensor\id)(\delta(x^\ell))=\sum_{i+j=n+k+\ell}\eps(x^i) \tensor x^j\neq x^\ell
    \ee
\end{Ex}
\noindent As we noted before, Frobenius algebras must be finite-dimensional, but Nearly Frobenius algebras can be infinite dimensional.
\begin{Ex}
    Let $A=K[[x,x\inv]]$, be the algebra of formal Laurent series over $K$. Coproducts of this algebraic structure may look like linear combinations of
    
    \be
    \delta_k(x^\ell)=\sum_{i+j=k+\ell}x^i\tensor x^j,
    \ee
    \noindent and each of these defines a Nearly Frobenius algebra which does not arise as a counit of a Frobenius algebra. We can observe that higher powers of these coproducts must take the form 
    \be
    \delta_k^{m}(x^\ell)=\sum_{j_1+j_2+\cdots j_{m-1}=mk+\ell}x^{j_1}\tensor \cdots \tensor x^{j_{m-1}}.
    \ee
\end{Ex}


\section{Topological Quantum Field Theory}\label{sec:TQFT}
\subsection{Background} In the late 1980s, Atiyah \cite{Atiyah} and Segal \cite{Segal} introduced the idea of $n$-dimensional topological quantum field theories to give a mathematical explanation to ideas derived in quantum physics. Here, we focus on the case where $n=2$. The relation between 2 dimensional TQFTs and Frobenius algebras was discovered by Dijkgraaf \cite{Dijkgraaf}, while Moore and Segal \cite{MS} proved an equivalence of categories between the category of finite-dimensional Frobenius algebras and the category of 2-dimensional TQFTs, see also Kock \cite{Kock} and Teleman \cite{Teleman_2011}.
\begin{Def}
	Let \cob denote the category of 2-cobordisms, whose objects are disjoint copies of the circle $S^1$ and whose morphisms are smooth oriented surfaces with boundary. More precisely, a morphism between two objects $N_0=\sqcup_n S^1$ and $N_1=\sqcup_m S^1$ is a Riemann surface $\Sigma_{g,n,m}$ with boundary $\partial\Sigma_{g,n,m}=N_0^{op}\sqcup N_1$ where the boundary components corresponding to $N_0$ are given one orientation and those corresponding to $N_1$ are given the opposite orientation.
\end{Def}  

\noindent We present the morphisms of the category \cob in terms of generators and relations. All cobordisms can be generated via composition and disjoint union by a set of five cobordisms: a cup with one incoming boundary component and no outgoing boundary components, a cup with one outgoing boundary component and no incoming boundary components, a cylinder with one incoming and one outgoing boundary component, a pair of pants with one incoming boundary component and two outgoing boundary components, and a pair of pants with two incoming and one outgoing boundary components. We then present relations that the compositions of these generators must follow.

\begin{figure}[ht]
    \centering
    \includegraphics[width=0.6\linewidth]{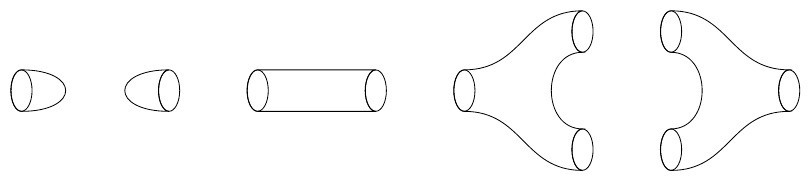}
    \caption{Generating set of morphisms in the category \cob}
    \label{fig:CobGens}
\end{figure}
\begin{Rmk}
    Often, we can consider a sixth generator known as the `twist'; however, this map is not necessary in the generation of all possible cobordisms. Instead, the twist satisfies some important relations. Two of these relations translate to commutativity and cocommutativity when we translate these ideas into the language of Frobenius algebras. The other relations (analogous to the Reidemeister moves in knot theory) involve the composition of the twist map with the other generators and turns \cob into a symmetric monoidal category. For these relations, we refer to \cite{Kock}.
\end{Rmk}
\begin{figure}[ht]
    \centering
    \includegraphics[width=0.4\linewidth]{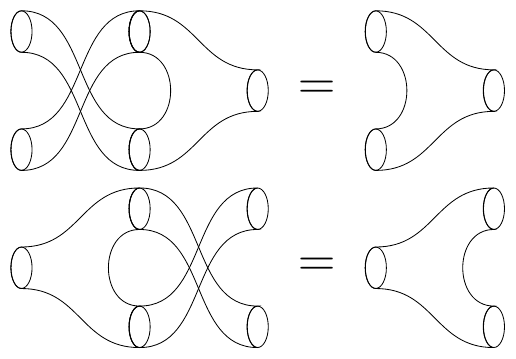}
    \caption{Commutativity and cocommutativity relations of the twist.}
    \label{fig:TwistComRel}
\end{figure}

\noindent Of the remaining relations on the other five generators, the first large family of them regards the fact that the cylinder acts as the identity morphism. As such, gluing a cylinder onto the incoming boundary or outgoing boundary of another generator (or in the case of the pairs of pants, gluing two cylinders onto the two boundary components of the same orientation) leaves that generator unchanged. Further, we introduce two relations regarding sewing discs into boundary components of pairs of pants. In particular, if we sew a disc into an incoming boundary component of a pair of pants with two incoming boundary components (or respectively to an outgoing boundary component of a pair of pants with two outgoing boundary components), then the resulting cobordism is a cylinder. Of course, sewing a disc itself does not look like the composition of cobordisms, so what we really mean is to sew a disc to one incoming (resp. outgoing) boundary component and an identity cylinder to the other.

\begin{figure}[ht]
    \centering
    \includegraphics[width=0.5\linewidth]{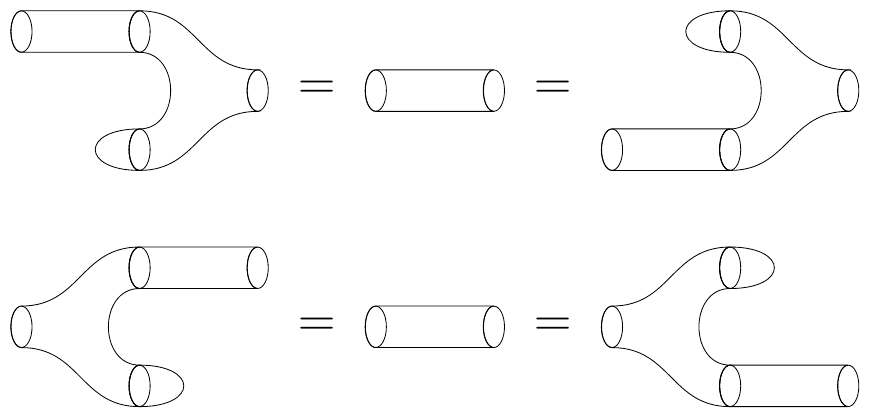}
    \caption{Relations involving sewing discs}
    \label{fig:sewingdiscs}
\end{figure}

\noindent Next, we have the associativity and coassociativity relations which regard the composition of two pairs of pants of the same orientation.

\begin{figure}[ht]
    \centering
    \includegraphics[width=0.3\linewidth]{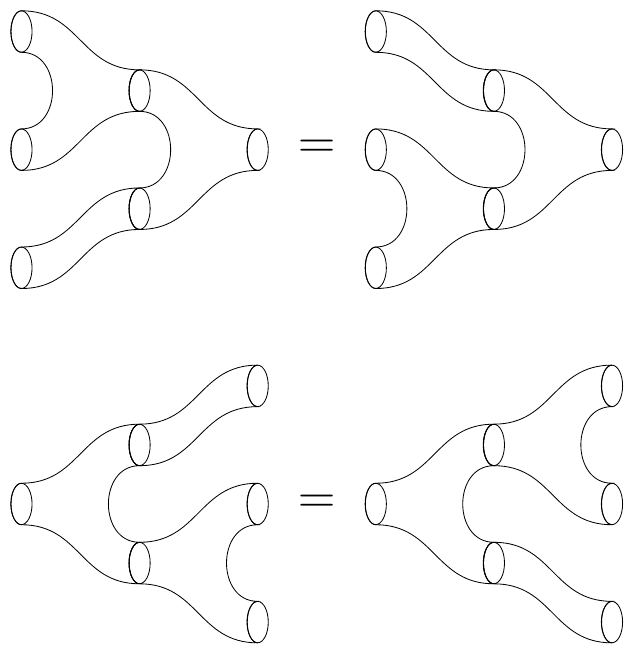}
    \caption{Associativity and coassociativity relations of cobordisms.}
    \label{fig:tqftassoc}
\end{figure}

\noindent Lastly, consider the following diffeomorphic Riemann surfaces of genus 0 with two incoming and two outgoing boundary components, each with different way of composing pairs of pants as opposite orientations. The final relation here, known as the Frobenius relation, dictates that they are also equivalent as cobordisms.

\begin{figure}[ht]
    \centering
    \includegraphics[width=0.6\linewidth]{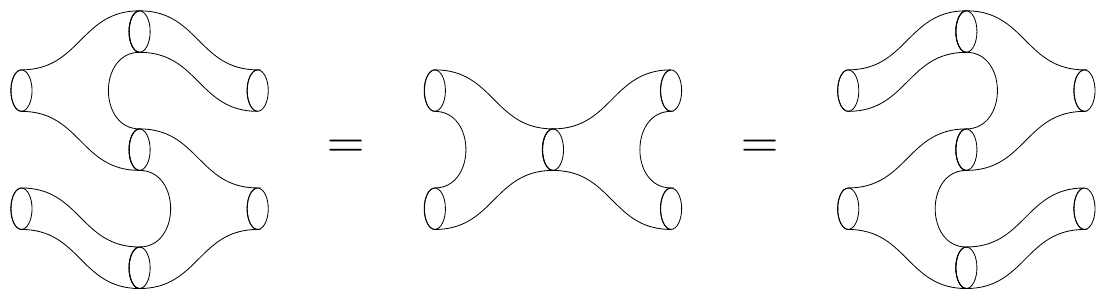}
    \caption{Frobenius relation of cobordisms.}
    \label{fig:FrobRel}
\end{figure}

We further consider the monoidal category of finite-dimensional vector spaces over a ground field $K$ together with a tensor product; we will call this category \textbf{Kvect}. Then, a 2-dimensional topological quantum field theory (2D TQFT) is a symmetric monoidal functor \be
Z:\mathbf{2Cob}\longrightarrow\mathbf{KVect},
\ee
satisfying a particular set of axioms. This functor sends the object in \cob consisting of $n$ copies of the circle to the $n$th tensor power of a particular vector space $A$ over $K$, and sends each Riemann surface $\Sigma_{g,n,m}$ to a multilinear map $\omega_{g,n,m}:\Atn\longrightarrow\Atm$. By functoriality, the relations on the generators of \cob correspond to the fact that this resulting vector space $A$ is necessarily a commutative Frobenius algebra with unit.

One can readily identify that the generators of \cob in Figure \ref{fig:CobGens} each correspond to a useful piece of data about the Frobenius algebra $A$.
\begin{align*}
	\eps&:=\omega_{0,1,0}:A\longrightarrow K,\\
	\mathbf{1}&:=\omega_{0,0,1}:K\longrightarrow A,\\
	\eta&:=\omega_{0,2,0}:A\tensor A\longrightarrow K,\\
	\delta&:=\omega_{0,1,2}:A\longrightarrow A\tensor A,\\
	m&:=\omega_{0,2,1}:A\tensor A \longrightarrow A.
\end{align*}

The two relations of the twist in Figure \ref{fig:TwistComRel} correspond to the commutativity and cocommutativity of $A$. The first relation on sewing discs in Figure \ref{fig:sewingdiscs} just defines $1$ as a multiplicative identity while the second gives the canonical basis expansion of $v$ of Equation \ref{eqn:canonbasis}. The two relations in Figure \ref{fig:tqftassoc} correspond to the associativity and coassociativity of $A$, while the Frobenius relation in Figure \ref{fig:FrobRel} is precisely the equation \ref{eqn:frobrel}.

Atiyah and Segal \cite{Atiyah} define the \textbf{sewing axiom} to express that the composition of morphisms must be respected by the 2D TQFT functor. That is, if a Riemann surface $\Sigma_{g+h+k-1,n,m}$ is built from two surfaces $\Sigma_{g,n,k}$ and $\Sigma_{h,k,m}$ by sewing the $k$ outgoing boundary components of one surface to the $k$ incoming boundary components of the other, then the multilinear map associated to the new surface should be the composition of the maps associated to the original two surfaces. In particular,
\be\label{eqn:sewing}
\omega_{g+h+k-1,n,m}=\omega_{h,k,m}\circ\omega_{g,n,k}:\Atn\longrightarrow\Atm.
\ee
\begin{figure}[ht]
    \centering
    \includegraphics[width=0.3\linewidth]{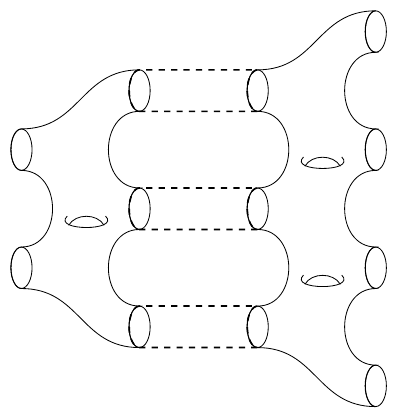}
    \caption{Sewing two cobordisms.}
    \label{fig:sewing}
\end{figure}

\noindent The sewing axiom of \ref{eqn:sewing} can then be generalized into considering an operation in which only some of these common boundary components are sewn together. In particular, consider two Riemann surfaces $\Sigma_{g,n,k}$ and $\Sigma_{h,\ell,m}$. Then for any $j\geq 1$ with $j\leq k$ and $j\leq\ell$, we can sew $j$ of the outgoing boundary components to the first surface to $j$ of the incoming boundary components of the second surface, and identity cylinders to the other boundary components of the same orientations. This results in a new Riemann surface $\Sigma_{g+h+j-1,n+\ell-j,k+m-j}$. The \textbf{partial sewing axiom} then dictates that the map associated to this new surface is also the composition of the multilinear maps associated to $\Sigma_{g,n,k}$ and $\Sigma_{h,\ell,m}$. That is to say, 
\be\label{eqn:partialsewing}
\omega_{g+h+j-1,n+\ell-j,k+m-j}=\omega_{h,\ell,m}\circ\omega_{g,n,k}:A^{\tensor n+\ell-j}\longrightarrow A^{k+m-j}.
\ee

 If we consider the partial sewing axiom for genus 0 surfaces with few boundary components, this allows us to examine relationships between important maps in the Frobenius algebra structure. For instance,
\be
\omega_{0,1,0}\circ \omega_{0,2,1}=\omega_{0,0,2}\Longrightarrow \eps\circ m=\eta.
\ee

\noindent In \cite{Atiyah} one more axiom insures the compatibility regarding duality and the reversal of orientations. In particular, if $Z(S^1)=A$, then the dual  algebra $A^*$ corresponds to $Z((S^1)^{op})$. Furthermore, let $\Sigma$ be a cobordism of type $(g,n,m)$ and $\Sigma^*$ be the surface $\Sigma$ with the opposite orientation. Then $Z(\Sigma^*)$ gives the dual map on dual vector spaces $(A^*)^{\tensor n}\longrightarrow (A^*)^{\tensor m}$. This compatibility regarding orientation gives $A$ a bialgebra structure. The nondegeneracy condition of  the bilinear form $\eta$, guarantees the duality between the algebra and coalgebra structures of $A$. This means that the map $\lambda$ in \ref{eqn:lambdaiso} gives an isomorphism between $A$ and its dual $A^*$, and makes $A$ a commutative Frobenius algebra.

\noindent We can see explicitly that the functoriality of $Z$ on the relations on the generating set in Figure \ref{fig:CobGens} lead to the properties intrinsic to Frobenius algebras. The relations on the twist in Figure \ref{fig:TwistComRel} imply the commutativity and cocommutativity of the algebra and coalgebra structures of $A$. The first relation depicted in Figure \ref{fig:sewingdiscs} gives that $1$ is the multiplicative identity of $A$ while the second relation implies \ref{eqn:canonbasis}. In particular, we can see \begin{align*}
    v=\id(v)&=(\id\tensor \eps)\circ \delta(v)\\
    &= (\id\tensor \eps)\left(\sum_{i,j,a,b}\eta(v,e_ie_j)\eta^{ia}\eta^{jb}e_a\tensor e_b\right)\\
    &=\sum_{i,j,a,b}\eta(v,e_ie_j)\eta^{ia}\eta^{jb}e_a \eps(e_b)\\
    &=\sum_{i,j,a,b}\eta(v,e_i\eps(e_b)\eta^{jb}e_j)\eta^{ia}e_a\\
    &=\sum_{i,a}\eta(v,e_i)e_a\eta^{ia} .
\end{align*}
\noindent The two relations in Figure \ref{fig:tqftassoc} correspond to the associativity and coassociativity of $m$ and $\delta$. Lastly, the Frobenius relation is precisely Figure \ref{fig:FrobRel} that can be expressed as the commutativity of some diagram described by the equation \ref{eqn:frobrel}.
Any Frobenius algebra gives rise to a 2D TQFT. Indeed, starting with the generating maps $1,\eps,\id, m, $ and $\delta$, build all other maps from the partial sewing axiom, subject to a few restrictions.

\begin{Def}[2-Dimensional Topological Quantum Field Theory]\label{tqft definition}
	 Let $A$ be a finite-dimensional vector space over a field $K$. Suppose $A$ is equipped with a non-trivial linear map $\eps:A\longrightarrow K$ and an isomorphism to its dual $\lambda:A\stackrel{\sim}{\longrightarrow} A^*$. A system of multilinear maps $\left(A,\left\{\omega_{g,n,m}\right\}\right)$ is a \textbf{2-dimensional topological quantum field theory} if the system $$\omega_{g,n,m}:\Atn\longrightarrow\Atm,\quad g,n,m\geq0$$satisfy the following four axioms.
    \begin{itemize}
        \item \textbf{TQFT 1. Symmetry:} The map \be
        \omega_{g,n,m}: \Atn\longrightarrow \Atm 
        \ee
        is invariant under the symmetric group action on its domain.
        \item \textbf{TQFT 2. Non-triviality: }\be
        \omega_{0,1,0}:=\eps:A\longrightarrow K.
        \ee

        \item \textbf{TQFT 3. Duality: }The following diagram commutes:
        \be
\begin{tikzcd}
A^{\otimes n} \arrow[dd, "\lambda^{\otimes n}"'] \arrow[rr, "{\omega_{g,n,m}}"] &  & A^{\otimes m} \arrow[dd, "\lambda^{\otimes m}"] \\
                                                                                &  &                                                 \\
(A^*)^{\otimes n} \arrow[rr, "{(\omega_{g,n,m})^*}"]                            &  & (A^*)^{\otimes m}                              
\end{tikzcd}
        \ee

    \item \textbf{TQFT 4. Partial Sewing: }\be
\omega_{h,\ell,m}\circ\omega_{g,n,k}=\omega_{g+h+j-1,n+\ell-j,k+m-j}:A^{\tensor n+\ell-j}\longrightarrow A^{k+m-j},\quad j\leq k,j\leq\ell.
    \ee
    \end{itemize}
\end{Def}

\noindent The maps $\omega_{g,n,m}$ can be classified for any 2D TQFT satisfying these four axioms. The two authors use the techniques of Cohomological Field Theories to give the proof this classification for all maps $\omega_{g,n,m}$ in \cite{DaDu} for $m>0$, extending the previous results obtained in \cite{DM_invitation} for maps of the form $\omega_{g,n,0}$. In this work, we will recall the main techniques used in \cite{DM_invitation} for the proof of the $m=0$ case. 

\noindent A \textbf{Cohomological Field Theory} (CohFT) \cite{KM} is a Frobenius algebra $A$ along with a system of linear maps $\Omega_{g,n}:A\longrightarrow H^*(\mgnbar,K)$ defined for $(g,n)$ with $2-2g+n>0$, and satisfying a particular list of axioms. Here, $\mgnbar$ stands for the Deligne-Mumford compactification of the moduli space of stable genus $g$ curves with $n$ marked points. It is known that the $\left\{\omega_{g,n,0}\right\}$ part of any 2D TQFT gives a CohFT that takes only values in $H^0(\mgnbar,K)=K$, so in the context of the 2D TQFT maps $\left\{\omega_{g,n,0}\right\}$, the CohFT axioms become the following:
\be\label{eqn:cohft1}
\omega_{g,n+1,0}(v_1,\ldots,v_n,1)=\omega_{g,n,0}(v_1,\ldots,v_n).
\ee
\be\label{eqn:cohft2}
\omega_{g,n,0}(v_1,\ldots,v_n)=\sum_{i,j}\omega_{g-1,n+2}(v_1,\ldots,v_n,e_i,e_j)\eta^{ij}.
\ee
\be\label{eqn:cohft3}
\omega_{g_1+g_2,|I|+|J|,0}(v_I,v_J)=\sum_{i,j}\eta^{ij} \omega_{g_1,|I|+1,0}(v_I,e_i)\cdot \omega_{g_2,|J|+1,0}(v_J,e_j).
\ee
where $I\sqcup J=\{1,2,\ldots,n\}$. Indeed, each of these three equations is a direct result of the partial sewing axiom \ref{eqn:partialsewing}. Since the multiplicative identity $1$ is the image of a cap $\Sigma_{0,0,1}$, then the first axiom \ref{eqn:cohft1} is a result of sewing a disc onto the last input boundary component of $\Sigma_{g,n+1,0}$. 
\begin{figure}[ht]
    \centering
    \includegraphics[width=0.2\linewidth]{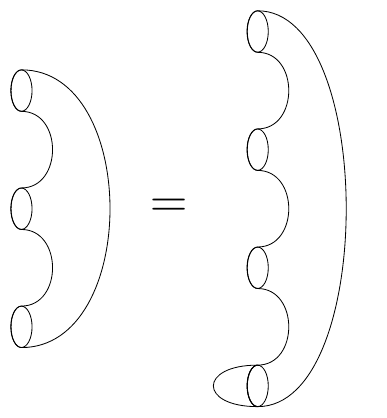}
    \caption{Partial sewing corresponding to the first CohFT axiom.}
    \label{fig:cohft1sewing}
\end{figure}

\noindent That is,
\be
\omega_{g,n+1,0}\circ(\id^{\tensor n}\tensor \omega_{0,0,1})=\omega_{g,n,0}:\Atn\longrightarrow K.
\ee

\noindent The other two CohFT axioms involve sewing on a tube-like shape $\Sigma_{0,0,2}$ associated to the element $\delta(1)=\omega_{0,1,2}\circ\omega_{0,0,1}$. We can see from \ref{eqn:deltabasis} that this is \begin{align*}
    \delta(1) &= \sum_{i,j,a,b}\eta(1,e_ie_j)\eta^{ia}\eta^{jb}e_a\tensor e_b\\
    &=\sum_{i,j,a,b}\eta(e_i,e_j)\eta^{ia}\eta^{jb}e_a\tensor e_b\\
    &=\sum_{i,j,a,b}\eta_{ij}\eta^{ia}\eta^{jb}e_a\tensor e_b\\
    &=\sum_{a,b}\eta^{ab}e_a\tensor e_b.
\end{align*}
\noindent If we sew this onto the the last two input boundary components of $\Sigma_{g-1,n+2,0}$, we create a hole adding 1 to the genus.
\begin{figure}[ht]
    \centering
    \includegraphics[width=0.15\linewidth]{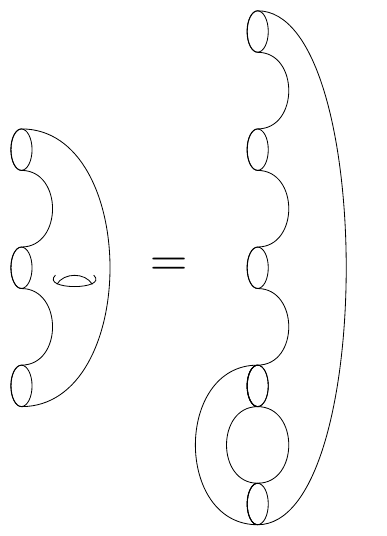}
    \caption{Partial sewing corresponding to the second CohFT axiom.}
    \label{fig:cohft2sewing}
\end{figure}

\noindent This gives the corresponding composition of linear maps:
\be
\omega_{g-1,n+2,0}\circ(\id^{\tensor n}\tensor \omega_{0,0,2}) = \omega_{g,n,0}:\Atn\longrightarrow K.
\ee

\noindent The third axiom also involves sewing on the same tube shape $\Sigma_{0,0,2}$, but instead sews one boundary comopnent each onto two disjoint cobordisms $\Sigma_{g_1,|I|+1,0}$ and $\Sigma_{g_2,|J|+1,0}$. Since the associated linear maps are invariant under the symmetric group action on their inputs, we will depict this as sewing onto the first input slot of one cobordism and the last input slot of the other cobordism, even though the formula \ref{eqn:cohft3} suggests we should sew these circles onto the last input slot of each.

\begin{figure}[ht]
    \centering
    \includegraphics[width=0.2\linewidth]{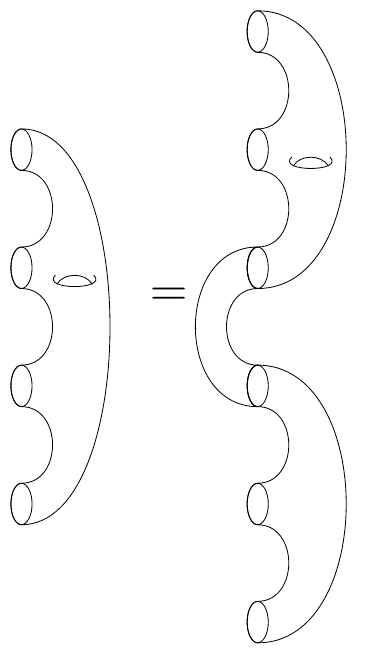}
    \caption{Partial sewing corresponding to the third CohFT axiom.}
    \label{fig:cohft3sewing}
\end{figure}

\noindent This gives a composition of $\omega_{0,0,2}$ applied to two separate linear maps at once.
\be
(\omega_{g_1,|I|+1,0}\tensor \omega_{g_2,|J|+1,0})\circ(\id^{\tensor |I|}\tensor \omega_{0,0,2} \tensor \id^{\tensor |J|})=\omega_{g_1+g_2,|I|+|J|,0}:A^{|I|+|J|}\longrightarrow K.
\ee
\noindent We now give the proof of the classification of 2D TQFT, but we will first establish a few base cases. By definition, we have that $\omega_{0,1,0}(v)=\eps(v)$ and $\omega_{0,2,0}(v,w)=\eta(v,w),$ but we appeal to duality to find $\omega_{0,3,0}$. Note that a linear map $\omega_{g,n+m,0}:\Atn\tensor\Atm\longrightarrow K$ is equivalent to a map $\Atn\longrightarrow (A^*)^{\tensor m}$, so we can reconstruct the map $\omega_{g,n,m}$ from $\omega_{g,n+m,0}$ using the following diagram:
\be\label{cd:reconstruct}
\begin{tikzcd}
A^{\otimes n} \arrow[dd,"="', rotate = 90] \arrow[rr, "{\omega_{g,n+m,0}}"] &  & (A^*)^{\otimes m} \arrow[rr, "(\lambda\inv)^{\otimes m}"] &  & A^{\otimes m} \arrow[dd, "="] \\
                                                                &  &                                                           &  &                               \\
A^{\otimes n} \arrow[rrrr, "{\omega_{g,n,m}}"]                  &  &                                                           &  & A^{\otimes m}                
\end{tikzcd},
\ee
where $\lambda:A\stackrel{\sim}{\longrightarrow}A^*$ is the isomorphism given by the Frobenius form $\eta(v,-).$ We can use this diagram to see that $vw=\omega_{0,2,1}(v,w)=\lambda\inv(\omega_{0,3,0}(v,w,-))$ or equivalently that $\omega_{0,3,0}(u,v,w)=\eta(uv,w)=\eps(uvw).$ This will serve as our base case for all maps of the form $\omega_{0,n,0}$.
\subsection{A classification result}
We will further review the classification result for $\omega_{g,n,0}$ as in \cite{DM_invitation} and \cite{DM_ribbon}

\begin{Prop}\label{prop:0n0case}
    The genus $0$ values of a 2D TQFT with no outputs is given by 
    \be\label{eqn:0n0case}
    \omega_{0,n,0}(v_1,v_2,\ldots,v_n)=\eps(v_1v_2\cdots v_n).
    \ee
\end{Prop}
\begin{proof}
    We've already seen the base cases for $n=1,2,$ and $3$. From here, we will use the third CohFT axiom \ref{eqn:cohft3} and the canonical basis expansion \ref{eqn:canonbasis}. Assume $n\geq 4$ and that \ref{eqn:0n0case} holds for all $\omega_{0,k,0}$ with $k<n$. Then,
    \begin{align*}
        \omega_{0,n,0}&=\sum_{a,b}\eta^{ab}\omega_{0,n-1,0}(v_1,\ldots,v_{n-2},e_a)\omega_{0,3,0}(v_{n-1},v_n,e_b)\\
        &=\sum_{a,b}\eta^{ab}\omega_{0,n-1,0}(v_1,\ldots,v_{n-2},e_a)\eta(v_{n-1}v_n,e_b)\\
        &=\sum_{a,b}\omega_{0,n-1,0}(v_1,\ldots,v_{n-2},\eta^{ab}\eta(v_{n-1}v_n,e_b)e_a)\\
        &=\omega_{0,n-1,0}(v_1,\ldots,v_{n-2},v_{n-1}v_n)\\
        &=\eps(v_1\cdots v_{n-1}v_n).
    \end{align*}
\end{proof}
\noindent We now observe that for higher genera, each added genus multiplies the inside of this function by another factor of the Euler element $\eul$.

\begin{Prop}\label{prop:g10case}
    For maps with positive genus $g>0$, we have
    \be\label{eqn:g10case}
    \omega_{g,1,0}(v)=\eps(v\eul^g).
    \ee
\end{Prop}
\begin{proof}
    We first use the second CohFT axiom to establish the base case $\omega_{1,1,0}.$
    \begin{align*}
        \omega_{1,1,0}(v)&=\sum_{a,b}\eta^{ab}\omega_{0,3,0}(v,e_a,e_b)\\
        &=\sum_{a,b}\eta^{ab}\eta(v,e_ae_b)\\
        &=\eta(v,\eul)\\
        &=\eps(v\eul).
    \end{align*}
    \noindent Now we use both the second and third CohFT axioms to induct on $n$.
    \begin{align*}
        \omega_{g,1,0}(v)&=\sum_{a,b}\omega_{g-1,3,0}(v,e_a,e_b)\eta^{ab}\\
        &=\sum_{i,j,a,b}\omega_{0,4,0}(v,e_a,e_b,e_i)\omega_{g-1,1,0}(e_j)\eta^{ij}\eta^{ab}\\
        &=\sum_{i,j,a,b}\eta(ve_ae_b,e_i)\omega_{g-1,1,0}(e_j)\eta^{ij}\eta^{ab}\\
        &=\sum_{i,j}\eta(v\eul,e_i)\omega_{g-1,1,0}(e_j)\eta^{ij}\\
        &=\omega_{g-1,1,0}(v\eul)\\
        &=\qquad\vdots\\
        &=\omega_{1,1,0}(v\eul^{g-1})\\
        &=\eta(v\eul^{g-1},\eul)=\eta(v,\eul^g)=\eps(v\eul^g).
    \end{align*}
\end{proof}
\noindent We now use a similar argument to fully generalize the case maps in a 2D TQFT with no outputs.
\begin{Thm}\cite{DM_invitation,DM_ribbon}\label{prop:gn0case}
    For maps whose outputs are in $K$, we have
    \be\label{eqn:gn0case}
    \omega_{g,n,0}(v_1,\ldots,v_n)=\eps(v_1\cdots v_n\eul^g).
    \ee
\end{Thm}
\begin{proof}
    The argument here is largely the same as that of \ref{prop:g10case}.
    \begin{align*}
        \omega_{g,n,0}(v_1,\ldots,v_n)&=\omega_{1,n,0}(v_1\eul^{g-1},v_2,\ldots,v_n)\\
        &=\sum_{a,b}\omega_{0,n+2,0}(v_1\eul^g,v_2,\ldots,v_n,e_a,e_b)\eta^{ab}\\
        &=\omega_{0,n,0}(v_1\eul^g,v_2,\ldots,v_{n-1},v_n\eul)\\
        &=\eps(v_1\cdots v_n\eul^g).
    \end{align*}
\end{proof}
 
 Theorem \ref{prop:gn0case} gives the 2D TQFT value for $m=0$. It is not difficult to see that this result can be extended to the case where $m\geq1$. For the $m=1$ case, applying $\lambda^{-1}$ once to this result give just $v_1\cdots v_n\eul^g.$ For higher outputs $m>1$, one can then make successive use of the comultiplication map $\delta$ to the one output case. The result given in \cite{thesis} and \cite{DaDu} is that for any Frobenius algebra, $A$, the $m\geq1$ portion of the 2D TQFT $\omega_{g,n,m}$ is given by

        \be\label{eqn:tqftclassification}
        \omega_{g,n,m}(v_1,\ldots,v_n)=\begin{cases}
            v_1\cdots v_n\eul^g,& m=1\\
            \delta^{m-1}(v_1\cdots v_n\eul^g),& m\geq 2
        \end{cases}.
        \ee

\subsection{Almost TQFT}
The notion of 2D TQFT can be generalized to Nearly Frobenius algebras, as introduced by Gonzalez et al. \cite{GLSU} as \textit{Almost TQFT}. If we consider, $\mathbf{2Cob^+}$, the full subcategory of \cob whose objects are a disjoint union of a positive number of copies of $S^1$ and $\mathbf{\textbf{KVect}^\infty}$, the category of possibly infinite-dimensional vector spaces over $K$, then an Almost TQFT is the functor \be
Z:\mathbf{2Cob^+}\longrightarrow\mathbf{\textbf{KVect}^\infty}.
\ee
It has been largely studied that the categories of Frobenius algebras and 2D TQFT are isomoprhic, and Gonzalez et. al similarly show that there is an isomorphism of categories between nearly Frobenius algebras and Almost TQFT.\\

In fact, this classification result in Equation \ref{eqn:tqftclassification} can be further extended to nearly Frobenius algebras and Almost TQFT, but only in the case that the algebra has a counit. The partial sewing techniques used in this setting can still be applied to a counital nearly Frobenius algebra; though, we note that for nearly Frobenius algebras which are not counital, we must develop the techniques found in Chapters \ref{SectionCCG} and \ref{SecInfiniteCCG}.
\begin{Thm} \cite{thesis, DaDu}\label{theorem A} 
	 The value of the Almost TQFT associated to a counital nearly Frobenius algebra $A$ with counit $\eps$, coproduct $\delta$, and Euler map $\bige$ is given by
	\be
	\omega_{g,n,m}(v_1,\ldots,v_n)=\begin{cases}
		\bige^g(v_1\cdots v_n),& m=1,\\
		\delta^{m-1}(\bige^g(v_1\cdots v_n)),& m\geq 2,\\
		\eps(\bige^g(v_1\cdots v_n)),& m=0.
	\end{cases}
	\ee
\end{Thm}

\section{Colored Cell Graphs and Edge Axioms}\label{SectionCCG}

\subsection{ Colored Cell Graphs}\label{chap:ccg}

In this chapter, we provide an alternate approach to the formulation of 2D TQFT. Dumitrescu and Mulase \cite{DM_invitation} introduced a formulation of the $m=0$ maps of a 2D TQFT using axioms on the category of cell graphs. Here, we expand upon these axioms by considering particular colorings of these cell graphs and introducing new axioms that respect this coloring.

We now consider coloring some subset of vertices of a cell graph. A colored cell graph of type $(g,n,m)$ is a cell graph of genus $g$ where $n$ of its vertices are labeled with one color, $m$ of its vertices are labeled with a different color, and the rest of the vertices remain uncolored. Of course, we could instead consider this as a cell graph where each of its vertices are labeled with one of three colors, but we would like for two of these colors to hold distinction.
\begin{Def}
    For a cell graph of type $(g,n,m)$ the $n$ vertices in one color are known as \textbf{input vertices} and the $m$ vertices in the second color are known as \textbf{output vertices}. Throughout this work, we label input vertices as red, output vertices as blue, and all other vertices as black.
\end{Def}
\begin{Rmk}
    In many contexts of graph coloring, one requires that two adjacent vertices cannot be labeled with the same color. We do not impose such a requirement here. Rather, this coloring of vertices induces a coloring on the edges of the graph. An edge connecting two input vertices or two output vertices share the same color as the two vertices it connects, and all other edges remain colorless.
\end{Rmk}
\noindent In addition this coloring of vertices inducing a coloring on the edges, this coloring also induces directions on (most of) the uncolored edges. If an uncolored edge is connected to an input vertex, then it points away from the input vertex, and if it is connected to an output vertex, then it points towards the output vertex. Edges between two uncolored vertices may point in either direction, so long as each uncolored vertex has at least one incoming edge and one outgoing edge.

\begin{Rmk}
Following these requirements, if we consider the subgraph containing all of the directed edges and those vertices connected to those edges, the input vertices are \textbf{sources}, the output vertices are \textbf{sinks}, and the uncolored vertices are known as \textbf{flow vertices}. We let $\Gamma_{g,n,m}$ denote the set of all cell graphs of genus $g$ with $n$ input vertices, $m$ output vertices, and some non-negative number of flow vertices. 
  \end{Rmk}

\begin{Def}[Ribbon TQFT]\label{def:ribbonTQFT}Let $A$ be a Nearly Frobenius algebra.  We define a {\bf ribbon TQFT} to be an assignment of a multilinear map compatible with the Edge Contraction/Construction Axioms.

    \be\label{eqn:CCGfunctor}
    \Omega: \Gamma_{g,n,m}\longrightarrow \Hom(\Atn,\Atm).
    \ee

Namely, for any colored cell graph $\gamma\in  \Gamma_{g,n,m}$ with $n,m\geq 1, $consider  $\Omega(\gamma):\Atn\longrightarrow\Atm.$
 an $n$-variable function $\Omega(\gamma)(v_1,\ldots,v_n)$ assigning $v_i\in A$ to the $i$-th input vertex of $\gamma$, so that $\Omega(\gamma)$ satisfies
 
 \begin{enumerate}
\item  Input Edge Contraction Axioms: Definition \ref{def:redaxioms}.
\item Output Edge Construction Axioms: Definition \ref{Output edge-construction axioms}. 
\item Flow Axioms and Base cases: Definition \ref{def:flowaxioms}.
\item Disc-bounding Loop Removal: If $\gamma, \gamma^\prime$ are graphs which only differ by a loop $L$ which bounds a disc, then we require that $\Omega(\gamma)=\Omega(\gamma^\prime)$ (see Fig. \ref{fig:loopremovalaxiom}).
 \end{enumerate}
\begin{figure}[ht]
	\centering
	\includegraphics[width=0.3\linewidth]{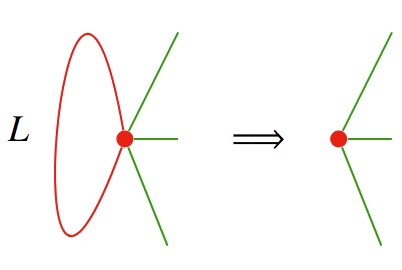}
	
	\caption{Axiom (4): Removal of a loop $L$ which bounds a disc}
	\label{fig:loopremovalaxiom}
\end{figure}

\end{Def}
	We note that in the case where $m=0$, the graph $\gamma$ has only input vertices, and therefore we only need to consider axioms (1) and (4) in the definition above. Further, in the case that $A$ is a Frobenius algebra, axiom (4) is redundant, as it is implied by the axiom (1) and the existence of a unit, and is given as Case 1 of Lemma \ref{lem:redloopremoval}.

  We claim that this assignment gives an alternate formulation of 2D TQFTs. Dumitrescu and Mulase \cite{DM_ribbon} give this formulation in the case of uncolored cell graphs and assignment of a map $\Omega(\gamma):\Atn\longrightarrow K$, to a genus $g$ cell graph with $m$ vertices. This formulation corresponds exactly to the $m=0$ case in our formulation. Inspired by their axioms on uncolored cell graphs, we introduce our own axioms that this assignment \ref{eqn:CCGfunctor} must satisfy for colored cell graphs.

\subsection{Edges between Input Vertices}\label{SecRedAxioms}
 The axioms listed in this subsection are similar to the edge-contraction axioms of Dumitrescu and Mulase \cite{DM_ribbon} in the uncolored case, but now we apply them only to edges connecting input vertices in our colored cell graph.
\begin{Def}[Input Edge-Contraction Axioms]\label{def:redaxioms} We say
	the assignment \ref{eqn:CCGfunctor} satisfies the following \textbf{input edge-contraction axioms} if

    \begin{itemize}
        \item \textbf{IECA 1:} Suppose there is an edge $E$ connecting the $i$-th and $j$-th input vertices of $\gamma$ where $i<j$ and $\gamma\in\Gamma_{g,n,m}$. Let $\gamma\pri\in\Gamma_{g,n-1,m}$ be the cell graph obtained by contracting $E$. Then,\be
        \Omega(\gamma)(v_1,\ldots,v_n)=\Omega(\gamma\pri)(v_1,\ldots,v_{i-1},v_iv_j,v_{i+1},\ldots,\widehat{v_j},\ldots,v_n),
        \ee
        where $\widehat{v_j}$ means we ignore the $j$-th input variable $v_j$, as its corresponding input vertex no longer exists in the graph $\gamma\pri$.
        \begin{figure}[ht]
            \centering
            \includegraphics[width=0.4\linewidth]{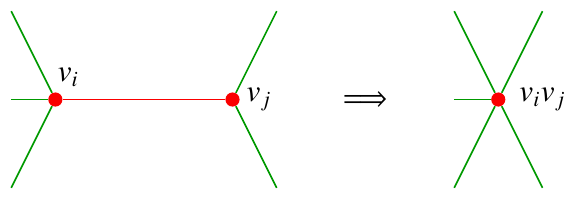}
            
            \caption{The input edge-contraction operation contracting an edge $E$ between input vertices $i$ and $j$}
            \label{fig:IECA1}
        \end{figure}
        \item \textbf{IECA 2:} Suppose there is a loop $L$ at the $i$-th input vertex of the cell graph $\gamma\in\Gamma_{g,n,m}$. Let $\gamma\pri$ be the potentially disconnected graph obtained by contracting the loop $L$ and separating the $i$-th input vertex into two distinct input vertices labeled $i$ and $i\pri$. For the purposes of labelling, we will order $i-1<i<i\pri<i+1$.
       \begin{figure}[ht]
            \centering
            \includegraphics[width=0.4\linewidth]{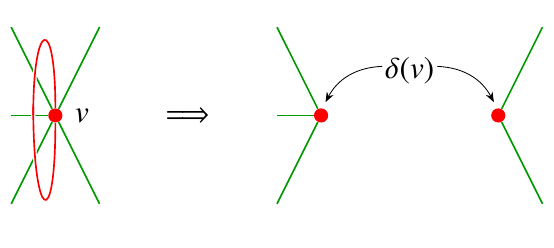}
            \caption{The input edge-contraction operation contracting a loop $L$ at input vertex $i$}
            \label{fig:IECA2}
        \end{figure}
        If $\gamma\pri$ is connected, then $\gamma\pri\in\Gamma_{g-1,n+1,m}$. In this case, we have that 
        \be
        \label{IECA2aeqn}
        \Omega(\gamma)(v_1,\ldots,v_n)=\Omega(\gamma)(v_1,\ldots,v_{i-1},\delta(v_i), v_{i+1},\ldots,v_n)
        \ee where the output of the comultiplication $\delta(v_i)$ is inserted into the $i$-th and $i\pri$-th input slots of the linear map.
        If, however, $\gamma\pri$ is disconnected, then write $\gamma\pri=(\gamma_1,\gamma_2)\in\Gamma_{g_1,|I|,m_1}\times\Gamma_{g_2,|J|,m_2}$ where $g=g_1+g_2,\, m=m_1+m_2,$ and $I\sqcup J=\{1,\ldots,\widehat{i},\ldots n\}$. We say $\gamma_1$ has input vertices with labels $\{v_I, v_i\}$ and $\gamma_2$ has input vertices with labels $\{v_J, v_{i\pri}\}$. Now let $(I_-,i,I_+)$ and $(J_-,i\pri,J_+)$ be reorderings of $I\sqcup\{i\}$ and $J\sqcup\{i\pri\}$ in increasing order. In this case, we have that
        \be
        \Omega(\gamma)(v_1,\ldots,v_n)=\sum_{a,b,k,\ell}\eta(v_i,e_ke_\ell)\eta^{ka}\eta^{lb}\Omega(\gamma_1)(v_{I_-},e_a,v_{I_+})\otimes\Omega(\gamma_2)(v_{J_-},e_b,v_{J+}),
        \ee which is similar to \ref{IECA2aeqn}, but since the output of the comultiplication $\delta(v_i)$ is inserted into input slots of two separate connected pieces, we write it in terms of a basis for $A$. Due to this tensor product being split across disconnected components, we assume that $A$ is cocommutative in this formula.
    \end{itemize}
\end{Def}
\subsection{Edges between Output Vertices}\label{SecBlueAxioms}

In this section, and throughout the remainder of this work, we introduce the following notation regarding the composition of maps between tensor powers of $A$. In particular, for a graph-induced TQFT $\Omega(\gamma)$, we define the compositions $(\delta \circ_{j} \Omega(\gamma))$ and $(m \circ_{i,j\mapsto i}\Omega(\gamma))$. We define composition $(\delta \circ_{j} \Omega(\gamma))$ by applying comultiplication to the $j$-th component of the output of $\Omega(\gamma):\Atn\longrightarrow\Atm$, and the result of that comultiplication will be the $j$-th and $j\pri$-th components of the resulting element of $A^{\tensor m+1}$. Similarly, we define $(m\circ_{i,j\mapsto i}\Omega(\gamma))$ by applying multiplication to the $i$-th and $j$-th components of the output of $\Omega(\gamma)$ and the result of that multiplication will be the $i$-th component of the resulting element of $A^{\tensor m-1}$ where the $j$-th component no longer exists. We show this definition more explicitly below for simple tensors and extend linearly to the whole of $\Atm$.

\begin{Def}
If $\Omega(\gamma)(v)=a_1\tensor a_2\tensor\cdots\tensor a_n$, then
\begin{align*}
	\delta \circ_{j} \Omega(\gamma)&:=a_1\tensor a_2\tensor\cdots\tensor a_{j-1}\tensor \delta(a_j)\tensor a_{j+1}\tensor\cdots\tensor a_n,\\
	m\circ_{i,j\mapsto i}\Omega(\gamma)(v)&:=a_1\tensor\cdots\tensor a_{i-1}\tensor a_ia_j\tensor a_{i+1}\tensor \cdots \tensor \widehat{a_j}\tensor\cdots\tensor a_n.
\end{align*}
\end{Def}
The following output edge construction axioms are a new set of rules about the assignments of a multilinear map to each colored cell graph $\gamma\in\Gamma_{g,n,m}$. These axioms are in some way dual to the input edge contraction axioms, and we apply them to the creation of edges connecting output vertices in our colored cell graph.
\begin{Def}[Output edge-construction axioms] \label{Output edge-construction axioms} We say the assignment \ref{eqn:CCGfunctor} satisfies the following \textbf{output edge construction axioms} if
    \begin{itemize}
        \item \textbf{OECA 1:} Suppose $\gamma\in\Gamma_{g,n,m}$ is a colored cell graph and $\gamma\pri$ is the graph obtained by splitting the $j$-th output vertex into two vertices $j$ and $j\pri$ with an edge connecting them (in such a way that if this edge between vertices $j$ and $j\pri$ of $\gamma\pri$ is contracted, then the resulting graph is $\gamma$). As before for labelling purposes, we will order $j-1<j<j\pri<j+1$. Then $\gamma\pri\in\Gamma_{g,n,m+1}$ and we impose that \be\label{eqn:deltacomp}
        \Omega(\gamma\pri)(v_1,\ldots,v_n)=\delta\circ_j \Omega(\gamma)(v_1,\ldots,v_n).
        \ee 
 
    \begin{figure}[htb]
        \centering
            \includegraphics[width=0.4\linewidth]{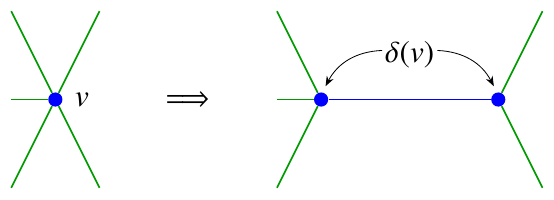}
            \caption{The output edge-construction operation splitting output vertex $j$ into two vertices with an edge between them}
            \label{fig:OECA1}
    \end{figure}
    \item \textbf{OECA 2a:} Suppose $\gamma\in\Gamma_{g,n,m}$ is a colored cell graph and $\gamma\pri$ is the graph obtained by taking the output vertices $i$ and $j$ of $\gamma$, combining them into the same vertex $i$, and then adding a loop at that vertex (such that if we were to contract this loop in $\gamma\pri$, the resulting graph would be $\gamma$). Then $\gamma\pri\in\Gamma_{g+1,n,m-1}$ and we impose that 

    \be
    \Omega(\gamma\pri)(v_1,\ldots,v_n)=m\circ_{i,j\mapsto i}\Omega(\gamma)(v_1,\ldots,v_n),
    \ee

    \item \textbf{OECA 2b: }This operation works the same way as \textbf{OECA 2a} if instead $\gamma=(\gamma_1,\gamma_2)\in \Gamma_{g_1,n_1,m_1}\times \Gamma_{g_2,n_2,m_2}$ is a disconnected with two connected pieces and $\gamma\pri\in\Gamma_{g,n,m}$ is created by combining one output vertex from each of the two connected pieces and adding a loop at the newly combined vertex (such that if we were to contract this loop in $\gamma\pri$, the resulting graph would be $\gamma$). In this case, we have $g_1+g_2=g,\, n_1+n_2=n, $ and $m_1+m_2=m$. We will relabel the input and output vertices of $\gamma_1$ as $\{v_1,v_2,\ldots v_{n_1},u_1,u_2,\ldots, u_{m_1}\}$ and the input and output vertices of $\gamma_2$ as $\{v_{n_1+1},v_{n_1+2},\ldots v_{n_1+n_2},u_{m_1+1},\\u_{m_1+2},\ldots, u_{m_1+m_2}\}$, and suppose that the vertices we combine are the output vertices $u_i$ and $u_j$ where $i\leq n_1<j$. In this case, we impose that
    \be
    \Omega(\gamma\pri)(v_1,\ldots,v_n)=m\circ_{i,j\mapsto i}\left(\Omega(\gamma_1)(v_1,\ldots,v_{n_1})\tensor\Omega(\gamma_2)(v_{n_1+1},\ldots,v_{n_1+n_2})\right).
    \ee
       \begin{figure}[htb]
        \centering
           \includegraphics[width=0.4\linewidth]{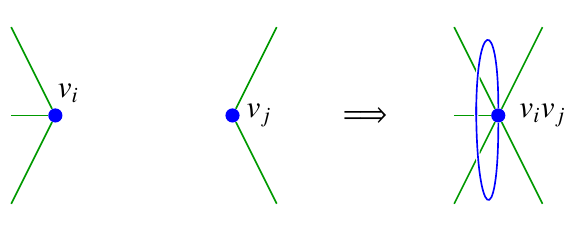}
            \caption{The output edge-construction operation combining output vertices $i$ and $j$ into a single vertex with an additional loop}
            \label{fig:OECA2}
    \end{figure}
       \end{itemize}
\end{Def}

\subsection{Topologically-inspired Axioms}\label{SecGreenAxioms}

Our previous axioms dealt with the contraction of edges between input vertices and the construction of edges between output vertices. In this section, we introduce the following topological equivalences on graphs, dealing with edges connected to uncolored vertices or vertices of differing colors. We begin by defining two topological moves on graphs and introduce what it means for two graphs to be equivalent under these moves.

In the case that a colored cell graph is actually a bipartite graph, then there are no edges between vertices of the same color and there are no flow vertices. Here, the only edges that exist are those connecting input vertices directly to output vertices, so we can consider them all as directed edges pointing from input vertices to output vertices.

\begin{Def}[Duplicate Edge Removal and Vertex Splitting/Directed Edge Contraction]	\label{VSDEC and DER}
	\noindent
	\begin{itemize}
		\item \textbf{Duplicate Edge Removal (DER): }Suppose $\gamma$ is a colored cell graph and $E_1, E_2$ are two edges contained in $\gamma$. If $E_1$ and $E_2$ are incident to the same two vertices and together bound a face of the graph, then consider the graph $\gamma\pri$ obtained by removing either of these two edges. In this case, we say that $\gamma$ and $\gamma\pri$ are equivalent under DER.
		\begin{figure}[ht]
			\centering
			\includegraphics[width=0.5\linewidth]{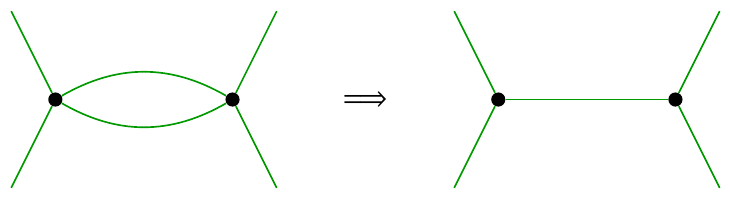}
			\caption{Duplicate edge removal.}
			\label{fig:DupEdge}
		\end{figure}

		\item \textbf{Vertex Splitting/Directed Edge Contraction (VSDEC): }Given a vertex $v$ with edges $E_1,E_2,\ldots, E_d$ in cyclic order, we can split the vertex $v$ into two new vertices $v$ and $v\pri$ with a directed edge $E\pri$ between them. We assign edges to these two vertices in the following way: For any choice of integers $i,j$ with $1\leq i\leq j\leq d$, the vertex $v$ gets assigned the edges $E_1,\ldots E_i, E\pri, E_{j+1},\ldots, E_d$ in cyclic order, and the vertex $v\pri$ gets assigned the edges $E\pri, E_{i+1},\ldots E_{j+1}$ in cyclic order (note that if $i=j$, then one of the two new vertices will be assigned only the edge $E\pri$). If $\gamma$ is a colored cell graph and $\gamma\pri$ is a graph created by applying a vertex splitting move to $\gamma$, then we say that $\gamma$ and $\gamma\pri$ are equivalent under VSDEC.
		
		\begin{figure}[ht]\label{VSDEC}
			\centering
			\includegraphics[width=.5\linewidth]{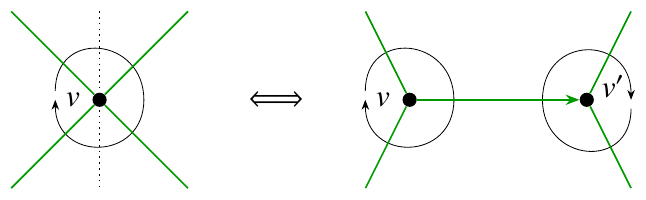}
			\caption{Vertex splitting move splitting a valence 4 vertex into two separate vertices.}
			\label{fig:VertSplit}
		\end{figure}
	\end{itemize}
\end{Def}

\begin{Rmk} 
	Note that if $v$ is an input vertex of Figure \ref{VSDEC}, then one of the two new vertices must remain an input vertex while the other becomes a flow vertex and the directed edge $E\pri$ points from the input vertex to a flow vertex. Similarly, if $v$ is an output vertex, then one of the two new vertices must remain an output vertex while the other becomes a flow vertex and the directed edge $E\pri$ points from the flow vertex to the output vertex. If $v$ is already a flow vertex, then both new vertices become flow vertices and either choice of direction of $E\pri$ is permissible.
\end{Rmk}
\begin{Rmk}
	
	The inverse operation of a vertex splitting operation is simply a contraction of a directed edge. We allow for the contraction of any directed edge that is not incident to both an input vertex and an output vertex; any directed edge we contract must be incident to at least one flow vertex. The resulting vertex after contraction will then take the type of the other vertex which is not the required flow vertex. That is, the contraction of an edge between an input vertex and a flow vertex will result in an input vertex, the contraction of an edge between a flow vertex and an output vertex will result in an output vertex, and the contraction of an edge between two flow vertices will result in a flow vertex.

	If $\gamma$ is a colored cell graph and $\gamma\pri$ is a graph created by contracting such a directed edge, then $\gamma$ can be equivalently obtained by performing a vertex splitting move to $\gamma\pri$. As such, we see that $\gamma$ and $\gamma\pri$ are equivalent under VSDEC.
\end{Rmk}

We next impose on the assignment $\Omega$ that it must respect the equivalence of the VSDEC and DER moves. Namely, if two graphs are equivalent under these moves, then they should give rise to the same multilinear map. The final assumption we will make on $\Omega$ is the particular multilinear maps assigned to some small foundational graphs. For the graph with one input vertex, no output vertices, and no edges, the graph with one output vertex, no input vertices and no edges, the graph with one input and one output vertex with a single edge between them, the graph with one input vertex connected to two output vertices by single edges, and the graph with one output vertex connected to two input vertices by single edges, we require that their associated multilinear maps are the counit, unit, identity map, comultiplication, and multiplication of the algebra, respectively. The following are the final two axioms dictating the assignment of a multilinear map $\Omega$ to a colored cell graph $\gamma$:
\begin{Def}[Flow Axioms and Base cases] \label{def:flowaxioms}
\noindent

	\begin{enumerate}
	 \item We impose that $\Omega(\gamma)=\Omega(\gamma\pri)$ whenever
	  $\gamma$ and $\gamma\pri$ are equivalent under either VSDEC or DER via Definition \eqref{VSDEC and DER}.
	 \item We further impose the following assignments for these five simple graphs:
	\begin{figure}[ht]
		\centering
		\includegraphics[width=0.5\linewidth]{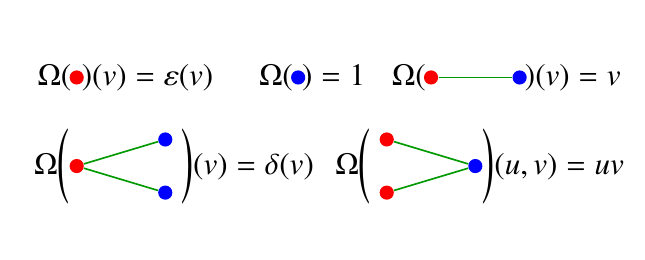}
		\caption{Simple graphs with associated maps corresponding to the counit, unit, identity, comultiplication, and comultiplication.}
		\label{fig:simplegraphs}
	\end{figure}
\end{enumerate}
\end{Def}
\begin{Rmk}
    The VSDEC and DER equivalences on directed edges not only allow us to transform every bipartite cell graph into a particular canonical form, but they also lead to particular topological consequences which share similar properties to those found in the structure of Frobenius algebras. In fact, the following are a list of operations such that if two graphs differ by one or more of these operations, then they are equivalent under VSDEC/DER.
\end{Rmk}
    
    \subsection{Frobenius Structures on Graphs}
    The following are three operations on ribbon graphs which resemble the relations on cobordisms in Section \ref{sec:TQFT}. Lemma \ref{lem:impliedmoves} indicates that these operations become redundant given the definitions of VSDEC/DER. Nevertheless, these operations are useful in practice in determining a canonical form of graphs of a given topological type.
    
\begin{Def}\label{def:topologicalmoves} (Implied operations) 
	
    \begin{itemize}
        \item \textbf{Valence 2 Flow Removal: }Let $\gamma$ be a directed cell graph, and let $v$ be a vertex of $\gamma$ with exactly one incoming and one outgoing vertex. Consider instead the graph obtained by removing $v$ and its two edges and adding a directed edge between the two vertices adjacent to $v$ in the same directions as before. This graph is the same as the graph obtained by simply contracting either of the edges connected to $v$. In short, we may remove any flow vertices which have valence exactly 2.
         \begin{figure}[ht]
    \centering
    \includegraphics[width=0.5\linewidth]{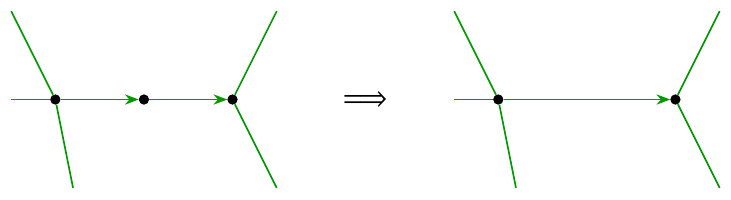}
    \caption{Removal of a flow vertex of valence 2 via edge contraction.}
    \label{fig:val2flow}
\end{figure}
        
        \item \textbf{Frobenius ``z=x'' Move: }Suppose a directed cell graph has a path of length 3 where the edges in the path are of alternating direction. Further assume that at the two vertices where two of these edges in the path meet, the edges are adjacent to each other in the cyclic order at these vertices. By splitting both of these two hinge vertices and contracting the edge between them, the valence of each hinge vertex is reduced by one and a new vertex of valence 4 is created. Through this move, a ``z-shaped'' or ``s-shaped'' subgraph of this type can be replaced with an ``x-shaped'' subgraph in a way that leaves the genus of the cell graph invariant and does not change the associated linear map.

        \begin{figure}[ht]\label{s shape}
    \centering
    \includegraphics[width=0.6\linewidth]{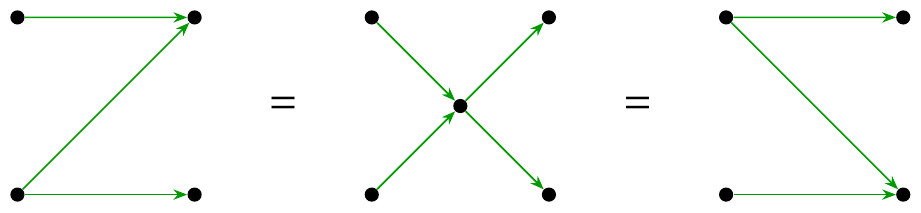}
    \caption{Equivalent subgraphs under the Frobenius ``z=x'' move.}
    \label{fig:zequalsx}
\end{figure}
        \item \textbf{Associativity/Co-associativity: }Splitting a vertex with three adjacent outgoing edges can be realized in two different ways. This bears a resemblance to the co-associativity of the pair of pants decomposition of a cobordism with one input hole and three output holes.
Similarly, splitting a vertex with three adjacent incoming edges can be realized in two different ways, bearing a resemblance to the associativity of the pair of pants decomposition of a cobordism with three input holes and one output hole. These can be both seen in Figure \ref{fig:coassocgraph}
\begin{figure}[ht]
    \centering
        \includegraphics[width=0.3\linewidth]{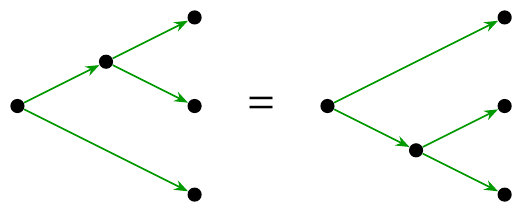}\hspace{1cm}
    \includegraphics[width=0.3\linewidth]{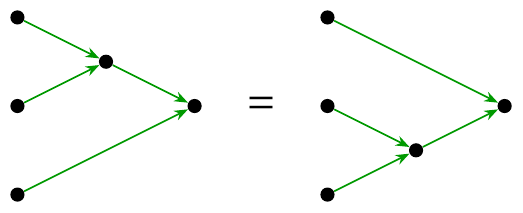}
    \caption{Equivalent subgraphs under co-associativity and associativity.}
    \label{fig:coassocgraph}
\end{figure}
    \end{itemize}
\end{Def}

        Note that the equivalence of subgraphs of Figure \ref{s shape} shares clear similarities with the Frobenius relation on 2-cobordisms as seen in Figure \ref{fig:FrobRel} and the corresponding relation on the multiplication and comultiplication in a Frobenius algebra (\ref{eqn:frobrel}).

\begin{Lem}\label{lem:impliedmoves}
	The VSDEC and DER equivalences in \ref{def:flowaxioms} imply the moves in Definition \ref{def:topologicalmoves}.
\end{Lem}
\begin{Cor}
	By the TQFT axioms of VSDEC and DER equivalences in \ref{def:flowaxioms}, the assignment $\Omega$ is constant under the topological moves defined in \ref{def:topologicalmoves}.
\end{Cor}
\subsection{Disc-bounding Loops and Duplicated Input/Output Edges}
Later we will consider the graph independence of this assignment $\Omega$ to a cell graph of type $(g,n,m),$ but first we will observe a simple case of this graph independence via two useful lemmas. 

\begin{Rmk}\label{rmk:loopremoval}
	Lemmas \ref{lem:redloopremoval} and \ref{lem:blueloopaddition} follow as a direct result of \textbf{IECA 2b} and \textbf{OECA 2b} in the setting of Frobenius algebras, provided that we make the assumptions that the graphs with only one vertex and no edges corresponding to the unit and counit of the Frobenius algebra. Later, when we consider potentially infinite dimensional Nearly Frobenius algebras, these algebras may not have a unit or counit, and in this case we take Case 1 of \ref{lem:redloopremoval} as axiom rather than a direct result of our previous axioms. We note that Cases 2 and 3 of \ref{lem:redloopremoval} follow directly from Case 1.
\end{Rmk}
\begin{Lem}\label{lem:redloopremoval}
    Let $\Omega$ be a Ribbon TQFT defined for a Frobenius algebra $A$. Suppose $\gamma\in\Gamma_{g,n,m}$. We consider three cases:
    \begin{itemize}
        \item \textbf{Case 1:} There exists a loop $L$ connected to an input vertex in $\gamma$ which bounds a disc, and $\gamma\pri\in\Gamma_{g,n,m}$ is the graph obtained by removing the loop $L$. \\(Note that we are not contracting the loop $L$, but simply removing it.)
        \item \textbf{Case 2:} There exist two edges $E_1$ and $E_2$ between two distinct input vertices $i$ and $j$ in $\gamma$ such that together $E_1$ and $E_2$ bound a face. $\gamma\pri\in\Gamma_{g,n,m}$ is the graph obtained by removing (not contracting) the edge $E_2$.
        \item \textbf{Case 3:} There exist two homotopic loops $L_1$ and $L_2$ in $\gamma$ both attached to the same input vertex $i$. $\gamma\pri\in\Gamma_{g,n,m}$ is the graph obtained by removing the loop $L_2$.
    \end{itemize}
    In each of these cases, we have that 
    \be\label{eqn:reddisc}
    \Omega(\gamma)(v_1,\ldots,v_n)=\Omega(\gamma\pri)(v_1,\ldots,v_n).
    \ee
\end{Lem}
\begin{figure}[ht]
    \centering
    \includegraphics[width=0.7\linewidth]{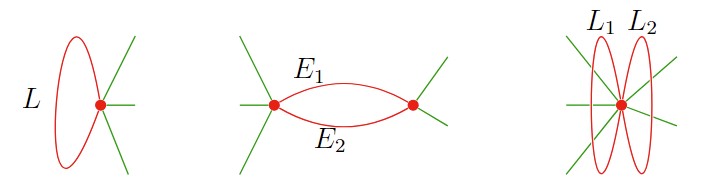}
    \caption{Removal of disc-bounding edges between input vertices.}
    \label{fig:reddisc}
\end{figure}
\begin{proof}
    In Case 1, note that contracting this disc-bounding loop $L$ at the $i$th input vertex of $\gamma$ creates a disconnected graph $(\gamma_0,\gamma\pri)\in\Gamma_{0,1,0}\times \Gamma_{g,n,m}$ where $\gamma_0$ is the simple graph with a single input vertex and no edges, and $\gamma\pri$ is the graph created by removing $L$. Here we compute the linear map associated to $\gamma$ via the edge contraction axiom \textbf{IECA 2b}.
    \begin{align*}
        \Omega(\gamma)(v_1,\ldots,v_n) &= \sum_{a,b,k,\ell}\eta(v_i,e_ke_\ell)\eta^{ka}\eta^{\ell b}\Omega(\gamma_0)(e_a)\Omega(\gamma\pri)(v_1,\ldots,v_{i-1},e_b,v_{i+1},\ldots,v_n)\\
        &=\sum_{a,b,k,\ell}\eta(v_i,e_ke_\ell)\eta^{ka}\eta^{\ell b}\eta(1,e_a)\Omega(\gamma\pri)(v_1,\ldots,v_{i-1},e_b,v_{i+1},\ldots,v_n)\\
        &=\sum_{b,\ell}\eta(v_i,1\cdot e_\ell)\eta^{\ell b}\Omega(\gamma\pri)(v_1,\ldots,v_{i-1},e_b,v_{i+1},\ldots,v_n)\\
        &=\sum_{b,\ell}\eta(v_i,e_\ell)\eta^{\ell b}\Omega(\gamma\pri)(v_1,\ldots,v_{i-1},e_b,v_{i+1},\ldots,v_n)\\
        &=\Omega(\gamma\pri)(v_1,\ldots,v_{i-1},v_i,v_{i+1},\ldots,v_n).
    \end{align*}
    For Case 2, note that contracting the edge $E_1$ turns the edge $E_2$ into a disc-bounding loop. This reduces to Case 1 in which we can simply remove this loop $E_2$. In this way, the graph obtained by contracting $E_1$ and removing $E_2$ is the same as the graph obtained by removing $E_2$ to reach the graph $\gamma\pri$ and then contracting $E_1$. Because of this, \ref{eqn:reddisc} holds.\\
    For Case 3, note that contracting the loop $L_1$  turns the loop $L_2$ into a disc-bounding loop. This reduces to Case 1 in which we can simply remove this loop $L_2$. As before, the graph obtained by contracting $L_1$ and then removing $L_2$ is the same as the graph obtained by removing $L_2$ to reach the graph $\gamma\pri$ and then contracting $L_1$. As such, \ref{eqn:reddisc} still holds.
\end{proof}
We now have a similar result regarding the addition of disc-bounding loops and faces to output vertices of a colored cell graph.

\begin{Lem}\label{lem:blueloopaddition}
    Suppose $\gamma\in\Gamma_{g,n,m}$. We consider three cases:
    \begin{itemize}
        \item \textbf{Case 1:} $\gamma\pri\in\Gamma_{g,n,m}$ is the graph obtained by adding a disc-bounding loop $L$ to the $i$th output vertex of $\gamma$. (Note that we are not changing the vertices of $\gamma$, but simply adding a loop.)
        \item \textbf{Case 2:} For some edge $E_1$ between two distinct output vertices $i$ and $j$ in $\gamma$, then  $\gamma\pri\in\Gamma_{g,n,m}$ is the graph obtained by adding the edge $E_2$ between the same two vertices such that together $E_1$ and $E_2$ bound a face.
        \item \textbf{Case 3:} For some loop $L_1$ connected to the $i$th output vertex of $\gamma$, then $\gamma\pri\in\Gamma_{g,n,m}$ is the graph obtained by adding a loop $L_2$ to the same vertex which is homotopic to $L_1$.
    \end{itemize}
    In each of these cases, we have that 
    \be\label{eqn:bluedisc}
    \Omega(\gamma)(v_1,\ldots,v_n)=\Omega(\gamma\pri)(v_1,\ldots,v_n).
    \ee
\end{Lem}
\begin{figure}[ht]
    \centering
    \includegraphics[width=0.7\linewidth]{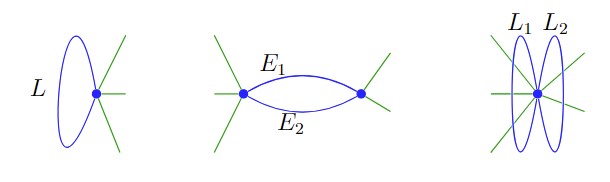}
    \caption{Addition of edge-bounding edges between output vertices}
    \label{fig:bluedisc}
\end{figure}
\begin{proof}
    In Case 1, we see that adding this disc-bounding loop $L$ at the $i$th output vertex of $\gamma$ is the same as taking the disconnected graph $(\gamma_0,\gamma)\in\Gamma_{0,0,1}\times\Gamma_{g,n,m}$ where $\gamma_0$ is the simple graph with a single output vertex and no edges and applying the edge construction axiom \textbf{OECA 2b} to the vertex of $\gamma_0$ and the $i$th vertex of $\gamma$. We use this axiom to compute the linear map associated to $\gamma\pri$.
    \begin{align*}
        \Omega(\gamma\pri)(v_1,\ldots,v_n)&=m\circ_{i,m+1\mapsto i}\Omega(\gamma)(v_1,\ldots,v_n)\otimes \Omega(\gamma_0)\\
        &=m\circ_{i,m+1\mapsto i}\Omega(\gamma)(v_1,\ldots,v_n)\otimes 1\\
        &=\Omega(\gamma)(v_1,\ldots,v_n).
    \end{align*}
    For Case 2, note that starting from a graph that splits an output vertex into vertex $i$ and vertex $j$ with an edge $E_1$ between them via \textbf{OECA 1} to reach the graph $\gamma$ and then adding and edge $E_2$ such that $E_1$ and $E_2$ bound a face is the same as adding $E_2$ as a disc-bounding loop and then splitting the vertex and adding $E_1$ via \textbf{OECA 1} to reach the graph $\gamma\pri$. Therefore, \ref{eqn:bluedisc} holds in this case.\\
    For Case 3, note that the loop $L_1$ can be created via \textbf{OECA 2} by combining to separate vertices into a single vertex with this loop $L_1$. If instead, one were to add a disc-bounding loop $L_2$ to one of these vertices and then perform \textbf{OECA 2}, then this gives the same graph as adding loop $L_1$ via \textbf{OECA 2} and then adding the homotopic loop $L_2$. Hence, this reduces to Case 1 and \ref{eqn:bluedisc} holds in this case as well.
\end{proof}
\begin{Rmk}
    Note that each of the three cases in Lemmas \ref{lem:redloopremoval} and \ref{lem:blueloopaddition} correspond to the removal or addition of valence 1 and 2 vertices in the ribbon graph dual to the colored cell graph $\gamma$.
\end{Rmk}
\section{Results on Colored Cell Graphs}\label{Section 6}
In this section, we show the graph independence of the 2D TQFT associated to colored cell graphs of type $(g,n,m)$. That is to say, if both $\gamma$ and $\gamma\pri$ are cell graphs of genus $g$ with $n$ input vertices and $m$ output vertices, then they should give rise to the same multilinear map $\Atn\longrightarrow\Atm$. To achieve this, we show that any 2-colored cell graph is equivalent via the edge contraction and construction axioms to some canonical form, potentially up to some small modifications. After formulating the map associated to a graph in this canonical form, we then see how we can reconstruct the map associated to the original graph via these axioms.

\subsection{Genus 0 Case}\label{gen zero case}
\begin{Def}\label{def:looplessnormal}
	(Normal form) We say a colored cell graph $\gamma\in\Gamma_{0,n,m}$ is in \textbf{normal form} if it has the following criteria: \begin{itemize}
		\item a single central flow vertex with $2$ incoming and $2$ outgoing edges,
		\item $n-2$ flow vertices with $2$ incoming edges and $1$ outgoing edge, where the incoming edges all connect either to vertices of this same type or input vertices, and the outgoing edge either connects to a vertex of the same type or the central flow vertex,
		\item $m=2$ flow vertices with $1$ incoming edge and $2$ outgoing edges, where the outgoing edges all connect either to vertices of this same type or output vertices, and the incoming edge either connects to a vertex of the same type or the central flow vertex,
		\item all input and output vertices have valence $1$.
	\end{itemize} That is, a genus $0$ cell graph is in normal form if it takes the form of the graph in Figure \ref{fig:bipartitetree}.
\end{Def}\begin{figure}[h]
\centering
\includegraphics[width=0.3\linewidth]{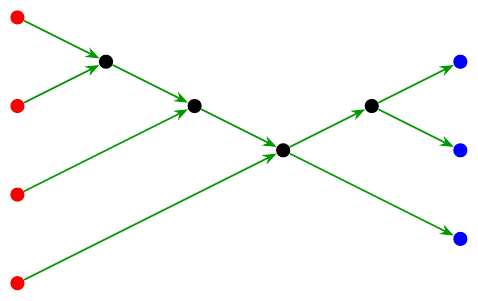}

\caption{A colored cell graph with $4$ input vertices, $3$ output vertices, $4-2$ `in-in-out` flow vertices, $3-2$ `in-out-out' flow vertices, and one valence $4$ flow vertex.}
\label{fig:bipartitetree}
\end{figure}
\begin{Def}\label{def:starshaped}
	(Star-shaped graph) We say a colored cell graph $\gamma\in\Gamma_{0,n,m}$ is \textbf{star-shaped} if each input and output vertex has valence $1$, and there is a single flow vertex connected to all other vertices. (See the graphs in figure \ref{fig:reorderleaves})
\end{Def}

\begin{Lem}\label{lem:bipartitetreeform}
    Consider a colored cell graph $\gamma\in\Gamma_{0,n,m}$ which is in normal form as defined in \ref{def:looplessnormal}. Then, 
    \be
    \Omega(\gamma)(v_1,v_2,\ldots,v_n)=\delta^{m-1}(v_1v_2\cdots v_n).\ee
\end{Lem}
\begin{proof}
    First note that if both $n=1$ and $m=1$ then this graph is exactly the basic graph $\gamma$ in \ref{fig:simplegraphs} for which $\Omega(\gamma)=\id$. 
    
    Now, if $n\geq 2$ and $m=1$, then we can successively contract flow edges and input edges via the directed edge contraction axiom and \textbf{IECA1} until we reach the simple graph with two input vertices and one output vertex. Each time we use \textbf{IECA1}, we multiply two of the inputs, and the map associated to this final graph is just the multiplication map. Following this process, we see the map associated to the original graph is simply the multiplication of the $n$ inputs $v_1\cdots v_n$.
           \begin{figure}[ht]
            \centering
            \includegraphics[width=.8\linewidth]{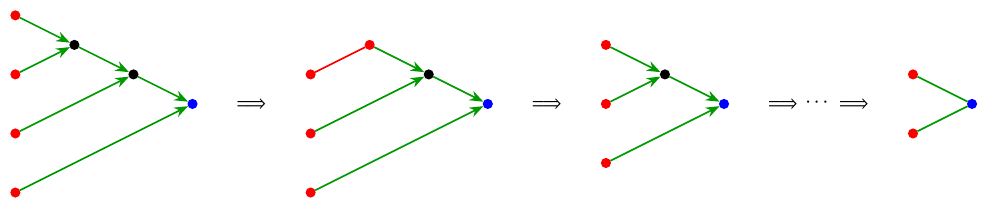}
            
            \caption{Proof of the $m=1$ case.}
            \label{fig:m1treeproof}
        \end{figure}\\
    Conversely, if $n=1$ and $m\geq 2$, then we can successively create flow edges and vertices and output edges and vertices via the vertex splitting axiom and \textbf{OECA1} starting from the simple graph associated to the comultiplication map. Following this process, each use of \textbf{OECA1} will add one output vertex and contribute an additional comultiplication to the resulting map. Therefore, the map associated to the original graph is the $(m-1)$-fold comultiplication of the single input, $\delta^{m-1}(v).$
    \begin{figure}[ht]
        \centering
        \includegraphics[width=.8\linewidth]{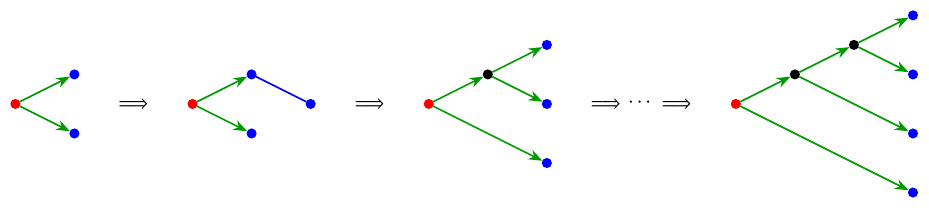}
        \caption{Proof of the $n=1$ case.}
        \label{fig:n1treeproof}
    \end{figure}\\
    Now, suppose both $n\geq 2$ and $m\geq 2$ and the algebra elements assigned to the input vertices are $v_1,\ldots, v_n$. Just as in the $m=1$ case, we can succesively contract directed edges and input edges until the graph is reduced to only one input vertex with the element $v_1v_2\cdots v_n$. This now reduces to the previous case where $n=1, m\geq2$ and so the associated map to this graph is $\delta^{m-1}(v_1v_2\cdots v_n)$.
    \begin{figure}[ht]
        \centering
        \includegraphics[width=\linewidth]{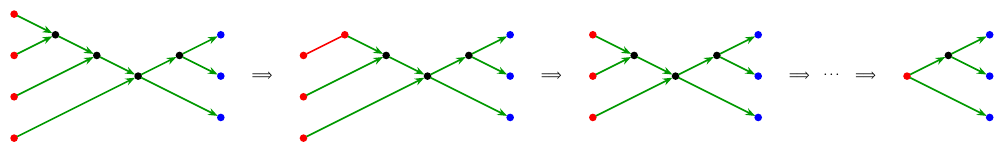}
        \caption{Reducing to the $n=1$ case.}
        \label{fig:fulltreeproof}
    \end{figure}
\end{proof}
We now show that every 2-colored cell graph with genus 0 is equivalent to a graph of the form described in Lemma \ref{lem:bipartitetreeform}.
\begin{Lem}\label{lem:reorderleaves}
    Let $\gamma\in \Gamma_{0,n,m}$ be a colored cell graph with $n$ input vertices and $m$ output vertices, all of valence 1, and a single central flow vertex connected to every other vertex. Via a series of edge contraction and vertex splitting moves, $\gamma$ is equivalent to a graph of the same form but with a cyclic order on the flow vertex such that all edges connecting to an input vertex are next to each other.
    \begin{figure}[ht]
        \centering
        \includegraphics[width=0.4\linewidth]{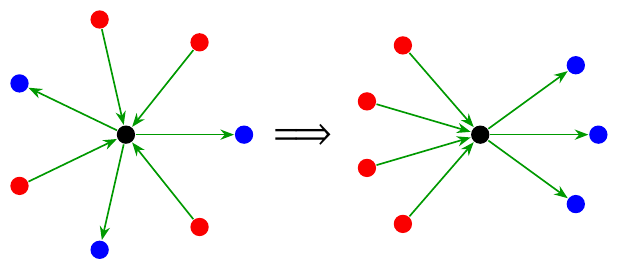}
        \caption{Reordering the leaves in a star-shaped graph}
        \label{fig:reorderleaves}
    \end{figure}
\end{Lem}
\begin{proof}
Let $\gamma\in\Gamma_{0,n,m}$ be a colored cell graph of the described form. Note that if we contract any edge, then the central vertex will become either an input vertex or an output vertex, depending on which edge we contract. Without loss of generality, we will contract edges connecting the central vertex to input vertices one at a time.\\
Choose one fixed input vertex and designate it $v_1$ and its edge $E_1$ connecting to the central vertex with edges in cyclic order $E_1,E_2,\ldots,E_d$. Now choose another input vertex $v_k$ and contract its edge $E_k$. The central vertex is now an input vertex with valence $d-1$ and edges in cyclic order $E_1,E_2,\ldots,\widehat{E_k},\ldots,E_d$. Now, we perform a vertex splitting move on this new central with $i=j=1$. This creates a new valence 1 input vertex which immediately follows the vertex $v_1$ in the cyclic order around the central vertex and returns the central vertex to being a flow vertex. If we repeat this process with all other input vertices, then the resulting graph will have all input vertices next to each other in the cyclic order of edges around the central vertex.
    \begin{figure}[ht]
        \centering
        \includegraphics[width=0.6\linewidth]{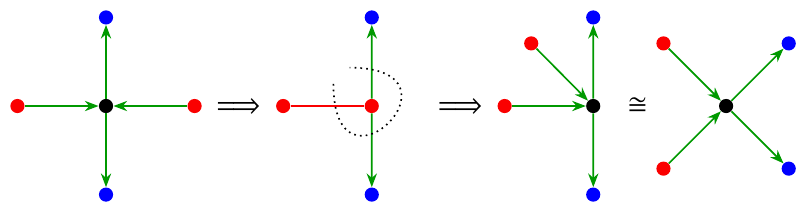}
        \caption{Repositioning an input leaf via edge contraction and vertex splitting}
        \label{fig:untangleoneleaf}
    \end{figure}
\end{proof}

\begin{Lem}\label{lem:staristree}
    Let $\gamma\in\Gamma_{0,n,m}$ be a star-shaped graph of the form of Lemma \ref{lem:reorderleaves} with $n$ valence 1 input vertices and $m$ valence 1 output vertices all connected to a single central flow vertex with the edges connected to input vertices all next to each other in the cyclic order. Then $\gamma$ is equivalent under the axioms to a graph satisfying the conditions of Lemma \ref{lem:bipartitetreeform}.
\end{Lem}
\begin{proof}
    In this star-shaped graph, the central vertex has valence $n+m$ while in the tree-shaped graph of Lemma \ref{lem:bipartitetreeform}, the central nexus vertex has valence 4. We can therefore achieve this transformation from the star-shaped graph to the tree-shaped graph by splitting the central vertex $n+m-4$ times, each time splitting off two edges at a time effectively reducing the valence of the central vertex by 1.

   \noindent Consider the position in the cyclic order around the central vertex where there is an edge connecting an input vertex adjacent to an edge connecting an output vertex (technically there are two such positions, but for our sake we will choose the one where the output edge is further clockwise than the input edge). Begin by splitting off the first two edges to input vertices, creating a new flow vertex. Next, split off from the central vertex the edge connected to the newest flow vertex and the next adjacent edge to an input vertex. After performing this same operation $n-2$ times, we have created a tree-shaped subgraph connecting the input vertices to the central vertex where each of the $n-2$ flow vertices created has exactly two incoming edges and one outgoing edge. 

   \noindent Next, split from the central vertex the first two edges to output vertices creating a new flow vertex. Then split off from the central vertex the edge connected to the newest flow vertex and the next adjacent edge to an output vertex. Again after performing this operation $m-2$ times, we have created a tree-shaped subgraph connecting the central vertex to the output vertices. Here each of the $m-2$ flow vertices has exactly one incoming edge and two outgoing edges. Since each of these splittings reduces the valence of the central vertex by 1, that vertex now has valence 4 and the graph is exactly of the form described in Lemma \ref{lem:bipartitetreeform}.
\end{proof}
\begin{figure}[ht]
    \centering
    \includegraphics[width=.8\linewidth]{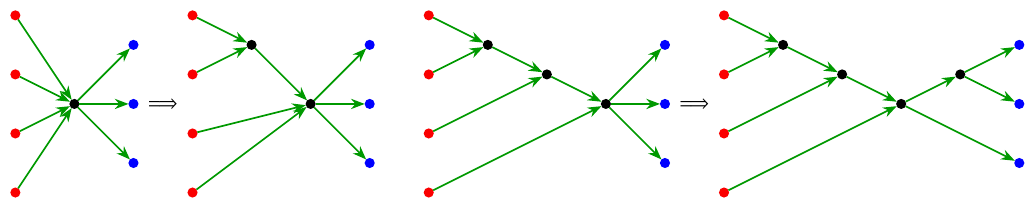}
    \caption{Vertex splitting moves converting a star-shaped graph to a tree-shaped graph.}
    \label{fig:spoketotree}
\end{figure}
\noindent We will now consider \textbf{bipartite} cell graphs, that is to say cell graphs with no flow vertices and only edges which directly connect an input vertex to an output vertex. We aim to show that every genus 0 bipartite cell graph is equivalent to a graph of the above canonical form, but first we make a simplification on the number of faces found in a particular bipartite cell graph. Recall that we can compute the genus of a cell graph $\gamma$ via the Euler characteristic $\chi(\gamma)$:
\be
\chi(\gamma)=V(\gamma)-E(\gamma)+F(\gamma)=2-2g,
\ee
where $V(\gamma), E(\gamma),$ and $F(\gamma)$ denote the number of vertices, edges, and faces of the graph, respectively. Both the tree-shaped graph described by Lemma \ref{lem:bipartitetreeform} and the star-shaped graph described in Lemma \ref{lem:reorderleaves} have only one face (the former has $2n+2m-3$ vertices and $2n+2m-2$ edges, the latter has $n+m+1$ vertices and $n+m$ edges, and in either case $2-2g=2$, requiring that the graph has only one face). Because of this, we would like to only consider bipartite cell graphs which have only one face to show that they are also equivalent to this canonical tree-shaped form. To do this, we employ the use of the Duplicate Edge Removal axiom.

\begin{Lem}\label{lem:onlyonefacebip}
    Let $\gamma\in\Gamma_{0,n,m}$ be a connected bipartite cell graph. Then there exists some spanning tree $\gamma\pri\in\Gamma_{0,n,m}$ which is a subgraph of $\gamma$ and contains only one face such that $\gamma$ and $\gamma\pri$ are equivalent under the \textbf{VSDEC} and \textbf{DER} axioms.
\end{Lem}
\begin{proof}
    If $\gamma$ is a graph with only one face, then $\gamma=\gamma\pri$ suffices here, so suppose $\gamma$ is a bipartite cell graph containing more than one face. Choose a face of this graph and consider the cycle of edges that bounds this face. Note that in a bipartite cell graph the only edges that exist are those between an input vertex and an output vertex. Therefore, if the cycle is to start and end at the same vertex, then it must have an even number of edges, and these edges are all alternating in direction traveling around the cycle. Any subpath of this cycle of length 3 therefore gives either a ``z-shaped'' or ``s-shaped'' subgraph that can be converted into the other shape via the corollary of \textbf{VSDEC} shown in Figure \ref{fig:zequalsx}. This move reduces the length of the cycle by 2. If we continue this process, we continue shortening the length of the cycle until the face of the graph is simply bounded by two duplicated edges. At this stage, we can remove one of these two edges by \textbf{DER}, thereby reducing $F(\gamma)$ by 1.
    \begin{figure}[ht]
        \centering
        \includegraphics[width=\linewidth]{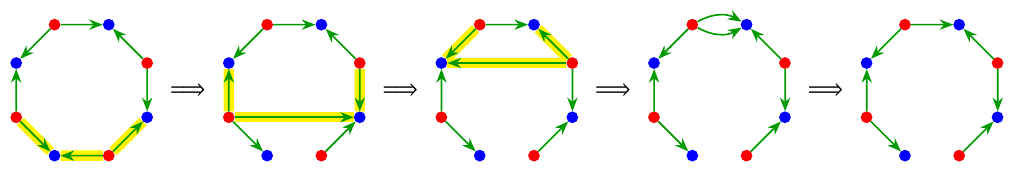}
        \caption{Reducing the length of a cycle via successive \textbf{VSDEC} moves until an edge can be removed via \textbf{DER}}
        \label{fig:movingz}
    \end{figure}

Up to our choice of length 3 subpath we begin this process with, this process is equivalent to simply removing one edge of the cycle. If we do this process for all but the last face of the graph, we can reduce the graph to one with just a single face which remains connected and has all of the same vertices, just with some of the edges removed. As this graph has only one face, it has no cycles and is the spanning tree $\gamma\pri$ that we desire.
\end{proof}
\noindent Now that every bipartite cell graph of genus 0 is equivalent to one with only one face, we only need to consider those with only one face.
\begin{Lem}\label{lem:onefaceisstar}
    Suppose $\gamma\in\Gamma_{0,n,m}$ is a connected bipartite cell graph with only one face. Then $\gamma$ is equivalent to a star-shaped graph described in Lemma \ref{lem:reorderleaves}.
\end{Lem}
\begin{proof}
    Let $\gamma\in\Gamma_{0,n,m}$ be a bipartite cell graph with only one face. At each input and output vertex, perform a vertex splitting move that leaves the input or output vertex as valence 1 and transfers all previously existing edges to the newly created flow vertex. This results in a subgraph of flow vertices which is isomorphic to the original graph $\gamma$ with each flow vertex connected to exactly one input or output vertex. This subgraph of flow vertices remains connected, so we can perform directed edge contraction moves until these flow vertices are combined into a single flow vertex. This single flow vertex is connected to each of the input and output vertices, and this graph is exactly the star-shaped graph described in Lemma \ref{lem:reorderleaves}.
    \begin{figure}[ht]
        \centering
        \includegraphics[width=0.5\linewidth]{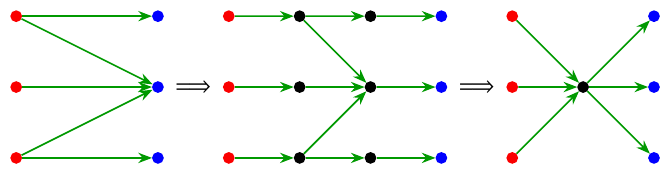}
        \caption{\textbf{VSDEC} moves converting a bipartite tree into a star-shaped graph}
        \label{fig:biptospoke}
    \end{figure}
\end{proof}

\noindent Coalescing the results of these previous lemmas, we can now find the linear map associated to any bipartite cell graph of genus 0.

\begin{Prop}\label{prop:bipartitegen0}
    Every connected genus 0 bipartite cell graph $\gamma\in\Gamma_{0,n,m}$ gives rise to the same map
    \be\label{eqn:gen0formula}
    \Omega(\gamma)(v_1,v_2,\ldots,v_n)=\delta^{m-1}(v_1 v_2\cdots v_n).
    \ee
\end{Prop}
\begin{proof}
    Suppose $\gamma\in\Gamma_{0,n,m}$ is a bipartite cell graph of genus 0. By Lemma \ref{lem:onlyonefacebip}, there exists some subgraph $\gamma_1$ which is a subgraph of $\gamma$ and is a spanning tree with only one face such that $\Omega(\gamma)=\Omega(\gamma_1)$. Then, by Lemmas \ref{lem:onefaceisstar} and \ref{lem:reorderleaves}, there exists some star-shaped graph $\gamma_2$ with $n$ valence 1 input vertices, $m$ valence 1 output vertices, and a single flow vertex connected to each of them in such an order that the edges connected to input vertices are all adjacent to each other. Lemma \ref{lem:staristree} shows that $\gamma_2$ is equivalent to some tree-shaped graph $\gamma_3$ with $n$ valence 1 input vertices, $m$ valence 1 output vertices, $n-2$ trivalent flow vertices with two incoming and one outgoing edge, $m-2$ trivalent flow vertices with two outgoing and one ingoing edge, and a single valence 4 flow vertex with two incoming and two outgoing edges. Therefore since $\Omega(\gamma)=\Omega(\gamma_3),$ then by Lemma \ref{lem:bipartitetreeform}, the formula \ref{eqn:gen0formula} holds.
    
\end{proof}

\subsection{Positive Genera}
\noindent We now consider the case where $g\geq 1$ by first observing how the addition of certain edges between input vertices or output vertices may change the map associated to a particular cell graph.

\begin{Def}
    (Petal) A \textbf{petal} is a pair of loops connected to the same vertex of a cell graph such that the four half-edges all appear adjacent in the cyclic order around the vertex, but neither loop itself bounds a face of the graph.
\end{Def}
\begin{figure}[ht]
    \centering
    \includegraphics[width=0.2\linewidth]{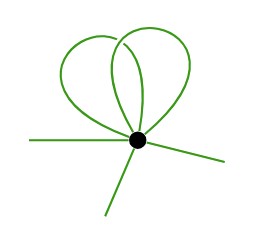}
    \caption{Vertex with a petal.}
    \label{fig:petaldef}
\end{figure}
\noindent In particular, if there are two loops $L_1$ and $L_2$ connected to the same vertex of a cell graph and we label the half-edges $L_1, \overline{L_1}, L_2, \overline{L_2}$, then these half-edges appear in the clockwise order $L_1,L_2,\overline{L_1},\overline{L_2}$. We can see that adding a petal to a cell graph increases the genus by 1. This corresponds to the idea that when embedding this graph onto a Riemann surface, in order for this petal to live on a surface and still maintain this cyclic order of half-edges without the two loops crossing, one of the loops must go around the hole of a surface while the other must go through the hole. Adding a petal to a graph requires an additional hole on the surface in order to prevent the loops from crossing.

\begin{Lem}\label{lem:redpetal}
    Adding a petal to an input vertex of a colored cell graph corresponds to multiplying by the vector associated to that vertex by the Euler element $\eul$. In particular, if $\gamma\in\Gamma_{g,n,m}$ is a colored cell graph and $\gamma\pri\in\Gamma_{g+1,n,m}$ is the graph obtained by adding a petal to the $i$th input vertex of $\gamma$, then
    \be\label{eqn:redpetal}
    \Omega(\gamma\pri)(v_1,\ldots,v_n)=\Omega(\gamma)(v_1,\ldots,v_{i-1},\eul v_i,v_{i+1},\ldots,v_n).
    \ee
\end{Lem}
\begin{proof}
    This follows from an application of \textbf{IECA 2a} followed by an application of \textbf{IECA 1} contracting the two loops that make the petal. In particular, we see that if $\gamma_1\in\Gamma_{g-1,n+1,m}$ is the graph created by one of the two loops in the petal, then
    \begin{align*}
        \Omega(\gamma\pri)(v_1,\ldots,v_n)&=\Omega(\gamma_1)(v_1,\ldots,v_{i-1},\delta(v_i),v_{i+1},\ldots,v_n)\\
        &=\Omega(\gamma_1)(v_1,\ldots,v_{i-1},m\circ\delta(v_i),v_{i+1},\ldots,v_n)\\
        &=\Omega(\gamma_1)(v_1,\ldots,v_{i-1},\eul v_i,v_{i+1},\ldots,v_n).
    \end{align*}
    Therefore, \ref{eqn:redpetal} holds.
\end{proof}
\begin{figure}[ht]
    \centering
    \includegraphics[width=0.7\linewidth]{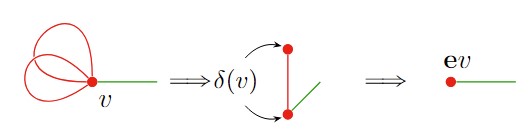}
    \caption{Contracting the two loops in a petal at an input vertex.}
    \label{fig:redpetal}
\end{figure}
\noindent A similar result holds if a petal is instead added to an output vertex.
\begin{Lem}\label{lem:bluepetal}
    Adding a petal to an output vertex of a colored cell graph corresponds to multiplying the component of the tensor product associated to that vertex by the Euler element $\eul$. That is, if $\gamma\in\Gamma_{g,n,m}$ is a colored cell graph and $\gamma\pri\in\Gamma_{g+1,n,m}$ is the graph created by adding a petal to the $j$th output vertex of $\gamma$, then 
    \be\label{eqn:bluepetal}
    \Omega(\gamma\pri)(v_1,\ldots,v_n)=(m\circ\delta) \circ_{j} \Omega(\gamma)(v_1,\ldots,v_n).
    \ee
\end{Lem}
\begin{proof}
    This similarly follows directly from an application of \textbf{OECA 1} followed by an application of \textbf{OECA 2a}. Let $\gamma_1\in\Gamma_{g,n,m+1}$ be the  obtained by applying \textbf{OECA 1} to the $j$th output vertex of $\gamma$, splitting it into two separate vertices. Then applying \textbf{OECA 2a} to the $j$th and $j\pri$th output vertices of $\gamma_1$ creates a petal to create the graph $\gamma\pri$. By the natural transformations associated to these two axioms, we see that
    \begin{align*}
        \Omega(\gamma\pri)(v_1,\ldots,v_n)&=m\circ_{j,j+1\mapsto j}\Omega(\gamma_1)(v_1,\ldots,v_n)\\
        &=m\circ_{j,j+1\mapsto j}\delta \circ_j \Omega(\gamma)(v_1,\ldots,v_n)\\
        &=(m\circ \delta)\circ_j \Omega(\gamma)(v_1,\ldots,v_n).
    \end{align*}
    Therefore \ref{eqn:bluepetal} holds.
\end{proof}
\begin{figure}[ht]
    \centering
    \includegraphics[width=0.7\linewidth]{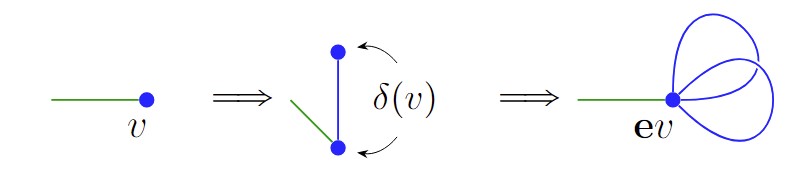}
    \caption{Adding two edges to create a petal at an output vertex.}
    \label{fig:bluepetal}
\end{figure}
\noindent These two previous lemmas largely cover the case where the genus is encoded by edges exclusively between input vertices or exclusively between output vertices. If the genus is encoded via flow edges, then we use a similar process to the techniques outlined in Section \ref{gen zero case} to reduce our graph to a star-shaped graph, but with additional petals at the flow vertex. Then, by using \textbf{VSDEC} moves, we can contract an edge to an input vertex and split this new vertex in such a way that the new input vertex acquires all of the petals from the flow vertex. From here, we simply apply \ref{lem:redpetal} to this new graph until we arrive at a star-shaped graph of genus 0.

\begin{Prop}\label{prop:bipartiteposgen}
    Every connected bipartite cell graph $\gamma\in\Gamma_{g,n,m}$ gives rise to the same map
    \be\label{eqn:bipartiteposgen}
    \Omega(\gamma)(v_1,v_2,\ldots,v_n)=\begin{cases}
        \eps(v_1\cdots v_n\eul^g),& m=0\\
        v_1\cdots v_n\eul^g,& m=1\\
        \delta^{m-1}(v_1\cdots v_n\eul^g),& m\geq 2
    \end{cases}.
    \ee
\end{Prop}
\begin{proof}
    Note that Dumitrescu and Mulase \cite{DM_ribbon} give the proof for the $m=0$ case, so we will consider $m\geq 1$.
    
    \noindent Now following \ref{lem:onlyonefacebip}, we can find some equivalent spanning tree which contains only one face. Then we follow the process of the proof of \ref{lem:onefaceisstar}, splitting each input and output vertex into a leaf and then contracting edges connecting flow vertices until there is just a single flow vertex. This results in a star-shaped graph with $2g$ additional loops at the central flow vertex. From here, we can reorder these leaves using the process of the proof of \ref{lem:reorderleaves} until we've partitioned the half-edges at the central vertex into three disjoint but adjacent groups: $n$ edges connecting the central vertex to input vertices, $m$ edges connecting the central vertex to output vertices, and $2g$ half-edges which arise as part of a loop.

    \noindent From here, we can perform a directed edge contraction from the central vertex to the input vertex adjacent to the collection of loops. We then split this new vertex in such a way that all of the loops now become attached to the input vertex, and the central vertex is connected to only input and output vertices. For the loops at the input vertex, we see by \ref{lem:redpetal} that contracting these loops via \textbf{IECA 2a} and \textbf{IECA 1}, will multiply the vector associated to this input vertex by $\eul^g$. What remains after this is a star-shaped graph $\gamma\pri$ whose associated input vectors are $v_1,v_2,\ldots,v_{i-1},\eul^g v_i,v_{i+1},\ldots v_n$ and therefore by \ref{lem:staristree} and \ref{lem:bipartitetreeform}, it follows that 
    \begin{align*}
        \Omega(\gamma)(v_1,v_2,\ldots,v_n)&=\Omega(\gamma\pri)(v_1,v_2,\ldots,v_{i-1},\eul^g v_i,v_{i+1},\ldots v_n)\\
        &=\delta^{m-1}(\eul^gv_1v_2\cdots v_n).
    \end{align*}
\end{proof}
\noindent Now that this holds for all bipartite cell graphs, we claim that this is precisely the classification for all colored cell graphs of type $(g,n,m)$. We will now give a proof of the Theorem \ref{thm B} for $A$ a Frobenius algebra and $m>0$.

\begin{Thm}\label{thm:cellgraphclassification}
Let $A$ be a Frobenius algebra and   $\gamma\in\Gamma_{g,n,m}$ a connected colored cell graph giving rise to a multilinear map $\Omega(\gamma)(v_1,\ldots, v_n) :\Atn\longrightarrow\Atm$ that satisfies the colored Edge Contraction Axioms. Then $\Omega(\gamma)$ is independent of the graph $\gamma$, depending only on the topology of the Riemann surface and it is given by
        \be\label{eqn:cellgraphclassification}
    \Omega(\gamma)(v_1,v_2,\ldots,v_n)=\begin{cases}
        \eps(v_1\cdots v_n\eul^g),& m=0\\
        v_1\cdots v_n\eul^g,& m=1\\
        \delta^{m-1}(v_1\cdots v_n\eul^g),& m\geq 2
    \end{cases}.
    \ee
\end{Thm}
\begin{proof}
    Again, we consider in our proof the case where $m\geq 1$. To find the map associated to a given colored cell graph $\gamma$, we consider the following process.

    \noindent First, we contract all edges between input vertices and use the input edge contraction axioms to record the input vectors of the new maps. If there is any positive genus encoded within these input edges, then some of these input vectors may be multiplied by powers of the Euler element $\eul$.

    \noindent Next, we contract all edges between output vertices and record the order in which we performed these moves. What remains after this is a disjoint union of connected bipartite graphs. Once we've resolved the outputs of this collection of graphs, we use the output edge reconstruction axioms to reconstruct the edges we just contracted. If there is any positive genus encoded within these output edges, then some of the slots in this $m$ tensor will be multiplied by $\eul$.

    \noindent For the disjoint union of connected bipartite graphs, we simply use Proposition \ref{prop:bipartiteposgen}. Throughout this whole process, between the input edges and the bipartite graphs, there will be $n-1$ instances of multiplying two inputs together, and between the output edges and the bipartite graphs, there will be $m-1$ instances of applying the coproduct. Additionally, the genus must be encoded either by input edges, by output edges, or by the bipartite graphs, but each of these scenarios corresponds to multiplication by a power of $\eul$ where the total power of these is $\eul^g$.

    \noindent Through this process, what remains is a sum over some number of basis elements of $m$ tensors with the each of the $n$ input vectors and $g$ powers of $\eul$ dispersed throughout the $m$ tensor factors. One can then verify up to coassociativity and the Frobenius relation that this is equivalent to $\delta^{m-1}(v_1v_2\cdots v_n).$
\end{proof}
\noindent Perhaps this proof of \ref{thm:cellgraphclassification} is most easily understood via a visual example.
\subsection{An Instructive Example}\label{sec:examplegraph}
In this section, we will consider a particular colored cell graph $\gamma$ of type $(2,2,2)$ and perform the process of finding the associated linear map $\Omega(\gamma):A\tensor A\longrightarrow A\tensor A$. Starting with the graph $\gamma$ in Figure \ref{fig:graphprocess}, we will label input and output vertices in these graphs in order from top to bottom.

\begin{figure}[htbp]
    \centering
    \includegraphics[width=.4\linewidth]{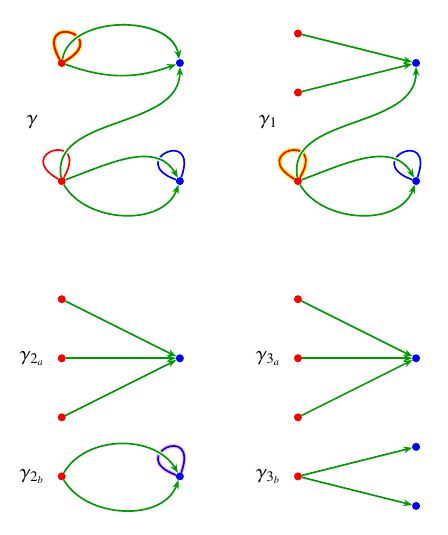}
    \caption{Contracting edges in a colored cell graph to determine the associated linear map}
    \label{fig:graphprocess}
\end{figure}

\noindent We begin by contracting the loop at the first input vertex via \textbf{IECA 2a}, separating that vertex into two other vertices. If the resulting graph is $\gamma_1$, then this axiom gives that
\be
\Omega(\gamma)(v_1,v_2)=\Omega(\gamma_1)(\delta(v_1),v_2).
\ee
\noindent Now, contracting the remaining input loop via \textbf{IECA 2b} splits the graph into two disjoint pieces $(\gamma_{2_a}, \gamma_{2_b})\in\Gamma_{0,3,1}\times\Gamma_{1,1,1}$, and this axioms says
\be\label{eqn:splitgamma1}
\Omega(\gamma_1)(\delta(v_1),v_2)=\sum_{a,b,k,\ell}\eta(v_2,e_ke_\ell)\eta^{ka}\eta^{\ell b}\Omega(\gamma_{2_a})(\delta(v_1),e_a)\tensor \Omega(\gamma_{2_b})(e_b).
\ee
\noindent Now, $\gamma_{2_a}$ is a genus 0 bipartite graph, so we have that \be\label{eqn:2a}
\Omega(\gamma_{2_a})(x,y,z)=xyz.
\ee
\noindent On the other hand, $\gamma_{2_b}$ has an output loop, and we contract this loop to get the resulting graph $\gamma_{3_b}$. If we were to reconstruct this loop via \textbf{OECA 2a}, then this axiom gives that
\be\label{eqn:2bto3b}
\Omega(\gamma_{2_b})(v)=m\circ \Omega(\gamma_{3_b})(v).
\ee
\noindent Now, $\gamma_{3_b}$ is a genus 0 bipartite graph, and we have that 
\be
\Omega(\gamma_{3_b})(v)=\delta(v).
\ee
\noindent Now applying \ref{eqn:2bto3b} gives that 
\be\label{eqn:2b}
\Omega(\gamma_{2_b})(v)=m\circ \delta(v)=\eul v.
\ee
We then use both \ref{eqn:2a} and \ref{eqn:2b} and apply them to \ref{eqn:splitgamma1} to see that 
\be
\Omega(\gamma_1)(\delta(v_1),v_2)=\sum_{a,b,k,\ell}\eta(v_2,e_ke_\ell)\eta^{ka}\eta^{\ell b}(m\circ\delta(v_1)\cdot e_a)\tensor(\eul e_b).
\ee
We then use the basis expansion \ref{eqn:canonbasis}, the expansion of $\delta$ \ref{eqn:deltabasis}, and the Frobenius relation \ref{eqn:frobrel} to compute this map.
\begin{align*}
    \Omega(\gamma)(v_1,v_2)&=\Omega(\gamma_1)(\delta(v_1),v_2)\\
    &=\sum_{a,b,k,\ell}\eta(v_2,e_ke_\ell)\eta^{ka}\eta^{\ell b}(m\circ\delta(v_1)\cdot e_a)\tensor(\eul e_b)\\
    &=\sum_{a,b,k,\ell}\eta(v_2,e_ke_\ell)\eta^{ka}\eta^{\ell b}(\eul v_1e_a)\tensor(\eul e_b)\\
    &=\sum_{a,b,k,\ell}\eta(v_2e_\ell,e_k)\eta^{ka}\eta^{\ell b}(\eul v_1e_a)\tensor(\eul e_b)\\
    &=\sum_{b,\ell}\eta^{\ell b}\eul v_1v_2e_\ell\tensor(\eul e_b)\\
    &=(\id\tensor m)\left(\sum_{b,\ell}\eta^{\ell b}\eul v_1v_2e_\ell\tensor e_b\tensor \eul\right)\\
    &=(\id\tensor m)(\delta(\eul v_1v_2)\tensor \eul)\\
    &=(\delta\tensor m)(\eul v_1v_2\tensor \eul)\\
    &=\delta(\eul^2 v_1v_2).
\end{align*}
\noindent We can see how this process generalizes: contracting input edges to apply the product or coproduct to input vectors, contracting output edges to decide how to apply the product and coproduct to the outputs of bipartite graphs, and then using the previously proven bipartite case to convert input vectors into output vectors. Each successive use of the coproduct followed by the product reduces the genus of the graph by one but in turn multiplies an input or output vector by the Euler element $\eul$.

\section{Colored Cell Graphs for Nearly Frobenius Algebras}\label{SecInfiniteCCG}
We claim that our formulation of 2D TQFT in \ref{thm:cellgraphclassification} using the colored cell graph approach also works as a formulation of Almost TQFT; however, we do need to apply a few modifications. The first consideration is that if the Nearly Frobenius algebra $A$ has no unit, then we have no interpretation of cell graphs with no input vertices, and if $A$ has no counit, then we have no interpretation of cell graphs with no output vertices. Because of this, we would like to consider graphs that at least one input and one output vertex. In particular, for the space $\Gamma_{g,n,m}$, we require $n,m\geq 1$.

\noindent Since Nearly Frobenius algebras may not be counital (and therefore may not have the bilinear form $\eta$), then we may not be able to write $\delta(v)$ in terms of a canonical basis expansion as we were able to in \ref{eqn:deltabasis}. Since the formula for \textbf{IECA 2b} relies on this, then we must change how this axiom is presented. The gist of the axiom is largely the same: the coproduct is applied to the vector at a particular input vertex and then the two factors of this coproduct are split across input vertices of two separate connected graphs. In light of this, consider the following. Suppose $v_i\in A$ is a vector such that $\delta(v_i)$ is can be written as a sum of tensors 
\be
\delta(v_i)=\sum_a u_a\tensor w_a.
\ee
\noindent In this case, the new \textbf{IECA 2b} should read that 
\be
\Omega(\gamma)(v_1,\ldots,v_n)=\sum_a \Omega(\gamma_1)(v_{I-},u_a,v_{I_+})\tensor \Omega(\gamma_2)(v_{J_-},w_a,v_{J_+}),
\ee

Before we proceed, we recall now the main result of \cite{DaDu} which was stated in Section \ref{sec:TQFT} as Theorem \ref{theorem A}.

\noindent Another major consideration here is regarding Lemmas \ref{lem:redloopremoval} and \ref{lem:blueloopaddition}. Recall the proofs of both of these lemmas involved contracting a disc bounding loop, splitting a graph into two disjoint pieces where one piece was just a graph with a single vertex. Since we no longer consider graphs of only one color (like a graph with only one vertex must be), this proof no longer applies. Instead, we must take the results of \ref{lem:redloopremoval} and \ref{lem:blueloopaddition} as additional assumptions that disc-bounding loops at input vertices may be removed and disc-bounding loops at output vertices may be added without changing the maps associated to the cell graph.

\noindent We provided a more procedural proof of \ref{thm:cellgraphclassification} rather than one that focused on the independence of the order of edge contraction and construction moves. The order of this procedure is even more critical in the case of Almost TQFT for the main reason that some cell graphs may contain loops that if contracted would split the graph into pieces where one of the connected pieces only has one color of vertices. See, for example, Figure \ref{fig:bad splitting} in which contracting this loop results in a disconnected graph where one of the disconnected pieces has no output vertices.

\noindent We claim we can avoid this problem if we require in our procedure that we always contract straight edges before we contract loops. This prevents us from splitting into a subgraph of only one color for the following reason. Without loss of generality, suppose we have a loop at an input vertex such that contracting it would cause one connected subgraph to have all the same color of vertices. By necessity, these must all be input vertices because one of the vertices caused by splitting the loop must both be input vertices. However, this means that in the original graph, all of the vertices inside the loop must be input vertices, and therefore they should be connected to the original input vertex via straight edges. If we instead contract the straight edges first, this original loop becomes a disc-bounding loop which we can remove by our new axiom.

\begin{figure}[ht]
    \centering
    \includegraphics[width=0.5\linewidth]{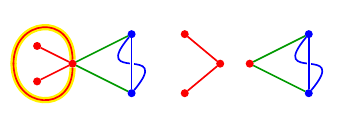}
    \caption{Contracting loops before straight edges could lead to disconnected graphs of only one color.}
    \label{fig:bad splitting}
\end{figure}

\noindent Following this procedure, we can present a similar classification result for Almost TQFT generated by colored cell graphs. Specifically, if $A$ is a Nearly Frobenius algebra and for $n,m\geq 1$, \be
\Omega:\Gamma_{g,n,m}\longrightarrow\Hom(\Atn,\Atm)
\ee
\noindent is an assignment of a multilinear map
\be
\Omega(\gamma):\Atn\longrightarrow\Atm
\ee
to a colored cell graph $\gamma\in\Gamma_{g,n,m}$ such that $\Omega$ satisfies the axioms of colored Edge Contraction and satisfies the results of Lemmas \ref{lem:redloopremoval} and \ref{lem:blueloopaddition}, then this assignment depends only upon the topological type $(g,n,m)$ of the graph. We are ready to prove the following result
\begin{Thm}\ref{thm B}\label{thm:infcellgraph}
    Let $A$ be a nearly Frobenius algebra with  coproduct $\delta$ and the Euler map $\bige$. Further let $\gamma$ be a colored cell graph of type $(g,n,m)$ with $n,m\geq 1,$ and $\Omega(\gamma)\in \Hom(A^{\tensor n}, A^{\tensor m})$ an assignment of a multilinear map satisfying  the Edge Contraction Axioms. Then the {\it ribbon TQFT} associated to $\gamma$ is independent of the graph and is given by

    \be
    \Omega(\gamma)(v_1,v_2,\ldots v_n)=\delta^{m-1}(\bige^g(v_1v_2\cdots v_n)).
    \ee
\end{Thm}
\begin{proof}
    The proof here follows the exact same procedure at that of \ref{thm:cellgraphclassification}. First we contract all straight edges between input edges followed by loops at input edges to record input vectors of new maps. Then we contract straight edges between output edges followed by loops at output edges and record the order in which we performed these moves to later reconstruct these outputs. We apply the same procedure of reducing bipartite graphs to trees with some number of loops and then use \textbf{VSDEC} to change the color of these loops. After contracting these loops, we are left simply with binary trees.

    \noindent As before, this corresponds to the application of the product and coproduct some appropriate number of times, which up to coassociativity and the Frobenius relation is precisely equivalent to $\delta^{m-1}(\bige^g(v_1\cdots v_n))$
\end{proof}

\noindent
We recall this result that was also mentioned in \cite{DaDu, thesis} .

\begin{Cor}\label{tqfts} Let $A$ be a counital nearly Frobenius algebra. Then the Almost TQFTs induced by the sewing axioms of Atiyah and Segal and the ribbon TQFTs induced by the Edge Contraction/Construction Axioms coincide. 
	\end{Cor}
	
	The proof relies on the classification Theorems \ref{theorem A} and \ref{thm:infcellgraph}.

\section{Twisted Catalan numbers recursion}\label{twisted Catalan}

\noindent Furthermore, in this section we
 exploit both classifications obtained in Theorems \ref{theorem A} and \ref{theorem A} to prove Theorem \ref{thm:coloredcatalan}, that is a generalization of Catalan numbers recursion relation twisted by TQFT of \cite{DM_invitation}  compatible with the formalism of topological recursion, algebra endowed with a nearly Frobenius structure $A$. 



\subsection{The recursion of Catalan numbers twisted by Almost TQFTs}

\noindent We return to our consideration of arrowed cell graphs and their enumeration by generalized Catalan numbers. We can think of cell graphs as a ``fattening'' of a graph where at each vertex, we prescribe a cyclic order of the half-edges incident to each vertex. We then give each vertex and each edge some thickness and embed the new graph onto a Riemann surface with minimal genus such that no edges intersect and at each vertex the incidents half-edges look clockwise according to the prescribed cyclic order.\\

\noindent We recall from Section \ref{Catalan}  that the generalized Catalan number

$$
C_{g,n}(\vec{\mu}):=|\widehat{\Gamma}_{g,n}(\vec{\mu})|$$
denote the number of uncolored arrowed cell graphs of type $(g,n)$ in  $\widehat{\Gamma}_{g,n}(\vec{\mu})$.

Since this Catalan recursion corresponds to edge contraction moves, and the TQFT formalism can be equivalently formulated via the same edge contraction moves, the Catalan recursion twisted by $\omega_{g,n}$ imply the following (\cite{DM_invitation} or  \cite{DM_ribbon}).

\noindent
\begin{Thm}\label{thm:coloredcatalan}
    Let $A$ be a counital nearly Frobenius algebra with counit $\eps$, Euler map $\bige$. Let $C_{g,n}(\vec{\mu})$ denote the number of arrowed cell graphs of type $(g,n)$ where the degrees of the vertices are specified by the vector $\vec{\mu}$. Further, consider the Almost TQFT $\omega_{g,n,m}:A^{\tensor n}\longrightarrow A^{\tensor m}$, and the restriction of that Almost TQFT given by $\omega_{g,n}(v_1,\ldots,v_n)=\omega_{g,n,0}(v_1,\ldots,v_n)=\eps(\bige^g(v_1\cdots v_n)).$ 
    
    Then, the Almost TQFT twisted by the generalized Catalan numbers satisfy a recursion relation given by

    \begin{align}\label{twistedtqft}
    	\begin{split}
        C_{g,n}(\vec{\mu})\cdot\omega_{g,n}(v_1,\ldots,v_n)=&
        \sum_{j=2}^n(\mu_1+\mu_j-2)C_{g,n-1}(\mu_1+\mu_j-2,\vec{\mu}_{[n]\setminus\{1,j\}})\cdot\omega_{g-1,n+1}(v_1v_j,\vec{v}_{[n]\setminus\{1,j\}})\\&+\sum_{\alpha+\beta=\mu_1-2}\alpha\beta\Biggl[C_{g-1,n+1}(\alpha,\beta,\vec{\mu}_{[n]\setminus1})\omega_{g-1,n+1}(\delta(v_1),\vec{v}_{[n]\setminus 1})\\&\hspace{70pt}+\sum_{\substack{g_1+g_2=g\\I\sqcup J=\{2,\ldots,n\}}}\sum_{a,b,k,\ell}\biggl(\eta(v_1,e_ke_\ell)\eta^{ak}\eta^{b\ell}C_{g_1,|I|+1}(\alpha,\vec{\mu_I})\\&\hspace{80pt}\times C_{g_2,|J|+1}(\beta,\vec{\mu_J})\omega_{g_1,|I|+1}(e_a,\vec{v}_I)\omega_{g_2,|J|+1}(e_b,\vec{v}_J)\biggr)\Biggr].
    \end{split}
\end{align}   
\end{Thm}

\begin{proof}
	For the nearly Frobenius algebra $A$, we restrict our Almost TQFT to the $m=0$ case giving us a map $\omega_{g,n}:A^{\tensor n}\longrightarrow K$.
	We recall that the Almost TQFT coincide with the ribbon TQFT by Corollary \ref{tqfts}.
	
	Therefore we need to consider cell graphs with only input vertices.
 We argue that the recursion \ref{twistedtqft} answers a counting problem of the number of arrowed cell graphs of type $(g,n)$ with degrees specified by $(\mu_1,\ldots \mu_n)$ where each vertex is weighted by a parameter $v_i\in A$ see also Section \ref{Catalan}, further preserving functoriality under the Input Edge Contraction Axioms defined in Section \ref{SecRedAxioms}. We notice that this counting problem gives a well defined map $A^n \longrightarrow K$ due to the finiteness of the number of cell graphs of type $(g, n)$ and degree $(\mu_1, \ldots, \mu_n)$ of Section \ref{Catalan}.
	
	We further observe by functoriality, that any cell graph of type $(g,n)$ and degree $(\mu_1,\ldots, \mu_n)$ with the weighting $(v_1, \ldots, v_n)$ will give rise to the same map $A^n \longrightarrow K$ given by the value of the Almost TQFT computed by the Theorem \ref{theorem A}. Therefore, our counting problem will result in multiplying the value of the Almost TQFT, $\eps(\bige^g(v_1\cdots v_n))$, by the number of such cell graphs. We recall from Section \ref{Catalan} that the count of arrowed cell graphs of type $(g,n)$ and degree $(\mu_1, \ldots, \mu_n)$ is the Catalan number
	$C_{g,n}(\mu_1, \ldots, \mu_n)$. Therefore this counting problem gives us the left hand side of Equation \ref{twistedtqft}. We further observe that we can answer this counting problem by contracting the arrowed edge or loop of a cell graph of type $(g,n)$ and degree $(\mu_1\ldots, \mu_n)$ at vertex weighted $v_1$ giving us the recursion formula of the right hand side of the Equation \ref{twistedtqft}.
	
	Indeed, we remark first that the three terms in the right hand side of the Equation \ref{twistedtqft} arise from the three possibilities of contracting the arrowed edge at the vertex weighted by $v_1$.

		 \begin{figure}[ht]
		\centering
		\includegraphics[width=0.4\linewidth]{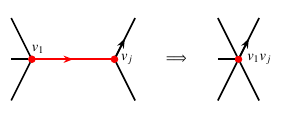}\hspace{1cm}
		\includegraphics[width=0.4\linewidth]{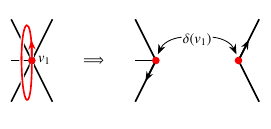}
		\caption{Edge contraction operations for arrowed cell graphs.}
		\label{fig:arrowECOgraph}
	\end{figure}
	
	The first term, of type $(g, n-1)$, arises from contracting an edge connecting the vertex weighted $v_1$ to the vertex weighted $v_j$. The second term, of type $(g-1,n+1)$, arises from contracting a loop at the vertex weighted $v_1$, given that contraction does not disconnect the graph. The last term, arises from contracting a loop at the vertex weighted $v_1$ in such a way that contraction disconnects the graph into of graph of type $(g_1, |I|+1)$ and another one of type $(g_2, |J|+1)$, where $g_1+g_2=g$ and 
	$I \sqcup J=\{2, \ldots, n\}$. As before, each of the three terms comes as a product between the number of the Almost TQFT and the number of arrowed cell graphs of the corresponding type, namely the associated Catalan number. We further recall from the Catalan recursion that contracting an edge or a loop of a cell graph changes the degree vector and the topological type of the associated cell, as in recursion \ref{cat} of  Section \ref{Catalan}.
	
	We will contract the edge of the cell graph with the arrowed half-edge at the first vertex weighted $v_1$. If this is an edge connecting to the $j$-th vertex weighted $v_j$, then the new vertex will be weighted $v_1v_j$, and we place an arrow on the half-edge of this new vertex which is next to the original contracted half edge in the cyclic clockwise ordering on the embedded surface. If instead, the edge we contract is a loop, the two new vertices will be weighted by the tensor components of $\delta(v_1)$, and we place arrows at each vertex at the half-edge next to the original loop in the cyclic ordering.

	 By contracting the arrowed edge/loop at vertex weighted $v_1$, we conclude that the counting problem gives a three terms recursion of the right hand side of Equation \ref{twistedtqft}, by the same argument used to give the Catalan recursion of \cite{DM_invitation, DM_ribbon}. It is useful to further remark that the Catalan recursion is based on edge or loop contractions, so we are not required to remove loops; therefore, Case 1 of Lemma \ref{lem:redloopremoval} is not needed.

\end{proof}

 In the next section, we introduce a vertex coloring on these graphs by labeling some of these as red vertices and the rest as blue vertices. This coloring of vertices induces a coloring on the edges where an edge connecting two vertices of the same color shares the color of the vertices, and an edge connecting two vertices of different colors can be labeled as green. These graphs are precisely the graphs of type $(g,n,m)$ discussed in chapter \ref{chap:ccg}, but with no flow vertices.

\subsection{Colored Catalan Numbers} 

In the this section, we propose a few ideas of avenues to be explored beyond this work.

We consider a Riemann Surface of genus $g$, and a collection of $n+m$ colored marked points, say that $p_1,\ldots, p_n$ red points while $q_1,\ldots, q_m$ blue points. A cell graph on this Riemann surface will have either edges or loops incident to each vertex $p_i$, and $q_j$, and the degree of a vertex will be the number of half edges attached to it. For example, a vertex incident to one loop and one edge will have degree 3, since a loop has 2 half edges. A red vertex can only be connected to a blue vertex by an edge, we will label it as a black edge.

We now consider arrowed cell graphs where in this setting we mark with an arrow one red half edge at each red vertex, and one blue half edge at each blue vertex. If no such half edges exist at a particular vertex then we leave that vertex unmarked. Let $\widehat{\Gamma}_{g,n,m}(\overrightarrow{\nu},\overrightarrow{d},\overrightarrow{\mu},\overrightarrow{f})$ be the set of arrowed cell graphs with $n$ red vertices and $m$ blue vertices and degrees specified by the four vectors $\overrightarrow{\nu}=(\nu_1,\ldots,\nu_n), \overrightarrow{d}=(d_1,\ldots,d_n), \overrightarrow{\mu}=(\mu_1,\ldots,\mu_m),$ and $ \overrightarrow{f}=(f_1,\ldots,f_m)$. We have $n$ red vertices and $m$ blue vertices, so we need to keep track of the number of half edges of each color we have: $\vec\nu$ gives the number of red half edges from a red vertex to other red vertices, $\vec{d}$ gives the number of black edges from red vertices to blue vertices, $\vec{\mu}$ gives the number of blue half edges from blue vertices to blue vertices, and $\vec{f}$ gives the number of black half edges from each blue vertex to red vertices. For example a graph $\gamma\in \widehat{\Gamma}_{g,n,m}(\overrightarrow{\nu},\overrightarrow{d},\overrightarrow{\mu},\overrightarrow{f})$ will be an arrowed cell graph on a surface of genus $g$ with $n$ red marked points $p_1,\ldots p_n$ and $m$ blue marked points $q_1,\ldots, q_m$, and with an arrow attached to one red or blue half edge at every vertex that such a half edge is available.
The graph $\gamma$ has $\nu_i$ red half edges and $d_i$ black half edges at each red vertex $p_i$ and further $\mu_j$ blue half edges and $f_j$ black half edges at each blue vertex $q_j$. We observe that $\sum_{i=1}^n d_i=\sum_{j=1}^m f_j$.

\begin{Lem}
Let $\widehat{\Gamma}_{g,n,m}(\overrightarrow{\nu},\overrightarrow{d},\overrightarrow{\mu},\overrightarrow{f})$ be the set of all arrowed colored cell graphs drawn on a closed connected oriented Riemann surface of genus $g$ with $n$ labeled red vertices, $m$ blue vertices and corresponding degrees at each vertex counted by the vector $(\overrightarrow{\nu}, \overrightarrow{d}, \overrightarrow{\mu}, \overrightarrow{f})$. Then $ \widehat{\Gamma}_{g,n,m}(\overrightarrow{\nu},\overrightarrow{d},\overrightarrow{\mu},\overrightarrow{f})$ is a finite set.   
\end{Lem}

\begin{proof}
Let $\gamma\in \widehat{\Gamma}_{g,n,m}(\overrightarrow{\nu},\overrightarrow{d},\overrightarrow{\mu},\overrightarrow{f})$ be a colored ribbon graph, and $c_{\alpha}(\gamma)$ be the number of $\alpha$ cells of a cell decomposition associated to each $\gamma$. Then the number of vertices 
$c_0(\gamma)=n+m$, the number of edges 
$$c_1(\gamma)=\dfrac{1}{2}\left(\sum_{i=1}^{n} \nu_i + \sum_{j=1}^m \mu_j +\sum_{k=1}^n d_k +\sum_{r=1}^{m} f_r\right)$$ and the number of faces $c_2(\gamma)$ satisfies the inequality
$$2-2g=c_0(\gamma)-c_1(\gamma)+c_2(\gamma).$$
Since the $c_0$ and $c_1$ are finite, we can have only finitely many graphs with a fixed number of vertices and edges. 

\end{proof}

\begin{Def}\label{cat def} We introduce {\it the colored Catalan number} as a count of colored arrowed cell graph 
$$C_{g, n, m}(\overrightarrow{\nu},\overrightarrow{d},\overrightarrow{\mu},\overrightarrow{f})
:=|\widehat{\Gamma}_{g,n,m}(\overrightarrow{\nu},\overrightarrow{d},\overrightarrow{\mu},\overrightarrow{f})|
$$
\end{Def}

\begin{Ex}
	There is only one unique cell graph of type $(0,1,1)$ with $k$ edges connecting the two vertices, since only one assignment of cyclic orientation results in this graph being able to be embedded on a sphere. Since there are no red or blue half edges then there are no arrows on this cell graph. Therefore, $C_{0,1,1}(0,k,0,k)=1$.

	 Further, one can see that $C_{0,1,1}(2a,1,2b, 1)=4ab\cdot C_a C_b$, where $C_m$ represents the classical Catalan number counting $m$ pairs of parentheses. This corresponds to $C_aC_b$ unarrowed graphs being made from connecting a graph of type $\Gamma_{0,1}(2a)$ with a graph of type $\Gamma_{0,1}(2b)$ then with $2a$ choices of an arrow at the first vertex and $2b$ choices of an arrow at the second vertex.
\end{Ex}

\begin{Quest}
    Do the {\it colored Catalan numbers} $C_{g,n,m}(\overrightarrow{\nu},\overrightarrow{d},\overrightarrow{\mu},\overrightarrow{f})$ satisfy a recursion relation similar to the Cut and Join equation of Hurwitz numbers?
\end{Quest}

    It serves as a challenging exercise to compute the base cases in the Question above.
     It is an interesting question in the area of Topological Recursion to further relate these recursions to enumerative geometry and Gromow-Witten theory establishing a rigorous mathematical foundation for these numbers with respect to intersection classes on moduli spaces of curves.

\nocite{*}
\setlength\bibitemsep{2\itemsep}
\printbibliography[title=REFERENCES]
\end{document}